\documentclass[11pt,oneside,english,reqno]{amsart}
\usepackage[utf8]{inputenc}
\usepackage[T1]{fontenc}
\usepackage{xcolor}
\usepackage{babel}
\usepackage{amstext}
\usepackage{wrapfig}
\usepackage{float}
\usepackage{amsthm}
\usepackage{amssymb}
\usepackage{graphicx}
\usepackage[unicode=true,pdfusetitle,
 bookmarks=true,bookmarksnumbered=false,bookmarksopen=false,
 breaklinks=false,pdfborder={0 0 0},pdfborderstyle={},backref=false,colorlinks=true]
 {hyperref}
\hypersetup{
 citecolor=blue}

\usepackage{geometry}
\makeatletter
\numberwithin{equation}{section}
\numberwithin{figure}{section}
\theoremstyle{definition}
\newtheorem{thm}{\protect\theoremname}[section]
  \theoremstyle{definition}
  
  \theoremstyle{definition}
  \newtheorem{defn}[thm]{\protect\definitionname}
  \theoremstyle{definition}
  \newtheorem{lem}[thm]{\protect\lemmaname}
  \theoremstyle{remark}
  \newtheorem{rem}[thm]{\protect\remarkname}
  \theoremstyle{definition}
  \newtheorem{notation}[thm]{\protect\notationname}
  \theoremstyle{definition}
  
  \theoremstyle{definition}
  \newtheorem*{thm*}{\protect\theoremname}
  \theoremstyle{definition}
  \newtheorem{prop}[thm]{\protect\propositionname}

\usepackage{geometry}
\usepackage{amsmath}
\usepackage{amsthm}

\numberwithin{equation}{section}
\numberwithin{figure}{section}
\usepackage{enumitem}		
 \let\footnote=\endnote
\@ifundefined{lettrine}{\usepackage{lettrine}}{}

\theoremstyle{definition}
\newtheorem{thmx}{Theorem}
\def\a{\alpha}

\def\d{\delta}
\def\s{\sigma}
\def\k{\kappa}

\def\s{\sigma}

\def\g{\gamma}

\def\w{\mathsf{w}}

\def\R{\mathbb{R}}

\def\N{\mathbb{N}}
\def\Z{\mathbb{Z}}

\def\ep{\varepsilon}

\usepackage{float}

\address{KIAS,
85 Hoegiro, Dongdaemun-gu,
Seoul, 02455,
Republic of Korea} 
\email{kihopark12@kias.re.kr}

\newcommand{\hol}{H\"older }
\newcommand{\vp}{\varphi}

\def\A{\mathcal{A}}
\def\AA{\mathsf{A}}
\def\B{\mathcal{B}}
\def\BB{\mathsf{B}}

\def\g{\gamma}
\def\gsf{\mathsf{g}}

\def\E{\mathcal{E}}

\def\P{\mathbb{P}}
\def\PP{\mathsf{P}}
\def\R{\mathbb{R}}
\def\RR{\mathsf{R}}

\def\CC{\mathcal{C}}

\def\W{\mathcal{W}}
\def\M{\mathcal{M}}

\def\U{\mathbb{U}}
\def\UU{\mathsf{U}}

\def\V{\mathbb{V}}

\def\v{\mathsf{v}}
\def\u{\mathsf{u}}
\def\vv{\mathsf{V}}
\def\WW{\mathbb{W}}
\def\ww{\mathsf{W}}
\def\HH{\mathsf{H}}
\def\id{\text{id}}

\def\eig{\text{eig}}

\newcommand{\wt}{\widetilde}
\newcommand{\ol}{\overline}

\def\ep{\varepsilon}

\def\mangle{\measuredangle}

\newcommand{\Wloc}{\mathcal{W}_{\text{loc}}}
\newcommand{\Sig}{\Sigma_T}

\newcommand{\glr}{\text{GL}_d(\R)}

\makeatother
  \providecommand{\corollaryname}{Corollary}
  \providecommand{\definitionname}{Definition}
  \providecommand{\factname}{Fact}
  \providecommand{\lemmaname}{Lemma}
  \providecommand{\propositionname}{Proposition}
  \providecommand{\remarkname}{Remark}
  \providecommand{\notationname}{Notation}
  \providecommand{\theoremname}{Theorem}
\providecommand{\theoremname}{Theorem}

\begin{document}

\sloppy
\title{Construction and applications of proximal maps for typical cocycles}
\author{Kiho Park}
\date{\today}
\begin{abstract}
For typical cocycles over subshifts of finite type, we show that for any given orbit segment, we can construct a periodic orbit such that it shadows the given orbit segment and that the product of the cocycle along its orbit is a proximal linear map. Using this result, we show that suitable assumptions on the periodic orbits have consequences over the entire subshift. 
\end{abstract}
\maketitle

\section{Introduction}

Given a continuous dynamical system $f \colon X \to X$, a \textit{matrix cocycle} is a continuous map $\A \colon X \to \glr$. For $x\in X$ and $n\in \N$, the product of $\A$ along the orbit of $x$ for time $n$ is denoted by
$$\A^n(x) :=\A(f^{n-1}x) \ldots\A(x).$$
Throughout the paper, our base dynamical system will be a mixing subshift of finite type $\s \colon \Sig \to \Sig$, a prototypical model for uniformly hyperbolic systems when equipped with the standard metric. 
In this paper, we will establish a useful property (Theorem \ref{thm: simultaneous}) for a class of cocycles, and derive some applications from it.

We will mainly focus on a subset of \hol continuous cocycles that are nearly conformal; we call them the \textit{fiber-bunched cocycles}. 
Typical cocycles are fiber-bunched cocycles with extra assumptions on some periodic point $p\in \Sig$ and one of its homoclinic points $z \in \Sig$; see Definition \ref{defn: typical} for the formal definition. We call this pair $(p,z)$ a \textit{typical pair}, and the typicality assumptions on the pair are suitable generalizations of the proximality and strong irreducibility in the setting of random product of matrices. Typicality has first been introduced by Bonatti and Viana \cite{bonatti2004lyapunov} to establish the simplicity of the Lyapunov exponents for any invariant measure with continuous local product structure. Later, it was shown that typical cocycles possess other interesting properties, including the uniqueness of the equilibrium states for the singular value potentials \cite{park2020quasi} and the validity of various limit laws \cite{park2020transfer, duarte2021h}.  The following theorem is another interesting property that can be derived from the typicality assumption. 

Recall that the \textit{singular values} of $g \in \glr$ are the square root of the eigenvalues of $gg^*$; we denote them by $\a_1(g) \geq \ldots \geq  \a_d(g)> 0$. Moreover, we denote its eigenvalues by $\{\eig_i(g)\}_{1 \leq i \leq d}$ so that $|\eig_i(g)| \geq  |\eig_{i+1}(g)|$ for all $1 \leq i \leq d-1$.
We then set
\begin{equation}\label{eq: notation for sing and eigen}
\vec{\mu}(g) := \big(\log\a_1(g),\ldots,\log\a_d(g)\big) \text{ and }\vec{\chi}(g):=\big(\log|\eig_1(g)|,\ldots,\log|\eig_d(g)|\big).
\end{equation}
\begin{thmx}\label{thm: simultaneous}
Let $\A \colon \Sig\to \glr$ be a typical cocycle. Then there exist $C>0$ and $k\in \N$ with the following property: for any $ x\in \Sig$ and $n\in \N$, there exists a periodic point $q \in \Sig$ of period $n_q\in [n,n+k]$ such that 
\begin{equation}\label{eq: singular and eigen}
\Big\|\vec{\mu}(\A^n(x)) - \vec{\chi}(\A^{n_q}(q))\Big\| \leq C.
\end{equation}
\end{thmx}

In the proof this theorem, which appears in Section \ref{sec: proof other thms}, we will see that a segment of the constructed periodic orbit $q$ shadows the forward orbit of $x$ upto time $n$ and the product $\A^{n_q}(q)$ is a proximal linear map. In such a point of view, this theorem resembles a result of Benoist \cite{benoist1997proprietes} (see also \cite{abels1995semigroups}) which states that for a reductive semi-group $\Gamma \subset \glr$, there exist a \textit{finite} subset $S\subset \Gamma$ and a constant $C>0$ such that for any $\g\in \Gamma$, there exists $s\in S$ satisfying 
$$ \Big\|\vec{\mu}(\g) - \vec{\chi}(\g s)\Big\| \leq C.$$
We will construct the required periodic point $q\in \Sig$ such a way that the exterior product cocycles of $\A^{n_q}(q)$ are all proximal, which will then allow us to compare its singular values and the eigenvalues. 
Similar ideas have repeatedly appeared in the literature for various purposes; see \cite{abels1995semigroups, benoist1997proprietes, sert2019large, kassel2020eigenvalue, breuillard2021joint} for instance.

Theorem \ref{thm: simultaneous} has many applications. 
Let $\M(\s)$ be the set of invariant probability measures  of $\s$ and $\E(\s)$ the set of ergodic measures of $\s$. 
For any $\nu \in \M(\s)$, we denote its Lyapunov exponents in a non-increasing order by $\{\lambda_i(\A,\nu)\}_{1 \leq i \leq d}$ and set
$$\vec{\lambda}(\nu):=\vec{\lambda}(\A,\nu)=(\lambda_1(\A,\nu),\ldots, \lambda_d(\A,\nu)).$$ For a periodic point $q\in \Sig$ of period $n_q\in\N$, its $i$-th Lyapunov exponent $\lambda_i(q):=\lambda_i(\A,q)$ is equal to $\displaystyle \frac{1}{n_q} \log\big|\eig_i(\A^{n_q}(q))\big|$. We then define
\begin{equation}\label{eq: Lyap periodic point}
\vec{\lambda}(q) := (\lambda_1(q),\ldots,\lambda_d(q)) = \frac{1}{n_q}  \vec{\chi}(\A^{n_q}(q)).
\end{equation}

The first application of Theorem \ref{thm: simultaneous} is that for typical cocycles, a uniform gap in the Lyapunov exponents over the periodic points implies the existence of a dominated splitting over the entire subshift $\Sig$ (see Definition \ref{defn: domination}):
\begin{thmx}\label{thm: main}
Let $\A \colon \Sig\to \glr$ be a typical cocycle. For some $1 \leq i \leq d-1$, suppose that there exists $c>0$ such that 
\begin{equation}\label{eq: assumption main}
\lambda_i(p)-\lambda_{i+1}(p)\geq c
\end{equation}
for every periodic point $p \in\Sig$.
Then $\A$ has the $i$-dominated splitting.  
\end{thmx}

This theorem partly generalizes recent results of Kassel and Potrie \cite[Theorem 3.1]{kassel2020eigenvalue} for locally constant cocycles and Velozo \cite[Corollary 1]{velozo2020characterization} for fiber-bunched $\text{GL}_2(\R)$-cocycles. 
While both results hold without further assumptions (such as the typicality) on the cocycle, the methods there do not readily generalize to fiber-bunched $\glr$-cocycles with $d\geq 3$. In our case, the proof of Theorem \ref{thm: main} makes uses of Theorem \ref{thm: simultaneous} which heavily relies on the typicality assumption. 

Theorem \ref{thm: simultaneous} has applications in thermodynamic formalism as well. Given two \hol continuous potential $\vp,\psi \colon \Sig \to \R$, their (unique) equilibrium states coincide if and only if there exists a constant $c\in \R$ such that the Birkhoff sum of the function $\vp-\psi-c$ vanishes on all periodic orbits; see \cite[Theorem 1.28]{bowen1975ergodic}. This raises the question whether the same can be said for matrix cocycles in the context of subadditive thermodynamic formalism, and we provide a partial answer to such a question in Theorem \ref{thm: subadditive} below.

Given a 1-typical (a weaker version of typicality; see Definition \ref{defn: typical}) cocycle $\A \colon \Sig \to \glr$, the author \cite{park2020quasi} showed that the norm potential $\Phi_\A$ has a unique equilibrium state $\mu_\A \in \E(\s)$ and that $\mu_\A$ has the subadditive Gibbs property; see Subsection \ref{subsec: subadditive thermo} for more details. For another 1-typical cocycle $\B$ whose unique equilibrium state $\mu_\B$ coincides with $\mu_\A$, it easily follows from the Gibbs property that there exists a constant $c\in \R$ (see Subsection \ref{subsec: subadditive thermo}) such that for every periodic point $p\in \Sig$,
$$\lambda_1(\A,p) - \lambda_1(\B,p) =c.$$
In fact, the constant $c$ is equal to $P(\Phi_\A) - P(\Phi_\B)$ where $P(\Phi)$ denotes the subadditive pressure of $\Phi$. The following theorem uses Theorem \ref{thm: simultaneous} to show that the converse statement holds when $\B$ is a small perturbation of $\A$:

\begin{thmx}\label{thm: subadditive} Let $\A \colon \Sig \to \glr$ be a 1-typical cocycle and $\B\colon \Sig \to \glr$ a small perturbation of $\A$. If there exists $c\in \R$ such that
\begin{equation}\label{eq: A,B}
\lambda_1(\A,p) - \lambda_1(\B,p) =c
\end{equation}
for every periodic point $p\in \Sig$, then 
$\mu_\A$ coincides with $\mu_\B$.
\end{thmx}

For irreducible locally constant cocycles, a similar result was recently established by Morris \cite[Theorem 9]{morris2018ergodic}: given two irreducible locally constant cocycles $\A,\B \colon \Sig \to \glr$, their equilibrium states are equal if and only if there exists a constant $c\in \R$ such that
$$\lambda_1(\A,\nu) -\lambda_1(\B,\nu) = c$$
for \textit{all} $\nu\in \M(\s)$. His proof easily extends to establish the same equivalence for fiber-bunched cocycles.
The content of Theorem \ref{thm: subadditive} is that with an added assumption of 1-typicality, the evidently weaker condition \eqref{eq: A,B} on only the periodic orbits is indeed sufficient for and equivalent to the condition that $\mu_\A$ is equal to $\mu_\B$.
Our proof, which relies on the typicality assumption, is different from that of Morris.

Lastly, Theorem \ref{thm: simultaneous} can also be used to compare different Lyapunov spectrums associated to typical cocycles $\A$. For any $x\in \Sig$ and $1 \leq i \leq d$, we denote the \emph{$i$-Lyapunov exponent of $x$} by $\displaystyle \lambda_i(x):=\lim\limits_{n\to \infty} \frac{1}{n}\log \a_i(\A^n(x))$, if the limit exists. 
If $\lambda_i(x)$ exists for all $1\leq i \leq d$, then we say $x$ is \emph{Lyapunov regular} and define the \emph{pointwise Lyapunov exponent} of $x$ by
$$\vec{\lambda}(x) := \lim\limits_{n\to \infty} \frac{1}{n}\vec{\mu}(\A^n(x))= (\lambda_1(x),\ldots,\lambda_d(x)) \in \R^d.$$
The \emph{pointwise Lyapunov spectrum} is then defined by $$\Omega_\A:= \{\vec{\lambda}(x)  \in \R^d \colon x \in \Sig \text{ is Lyapunov regular}\}.$$ 

While $\Omega_\A$ may be quite irregular for arbitrary cocycles, the author showed in \cite{park2020quasi} that $\Omega_\A$ is a closed and convex subset of $\R^d$ for typical cocycles $\A$ by generalizing the result of Feng \cite{feng2009lyapunov}. In particular, this implies that the closure of the \emph{Lyapunov spectrum over periodic points} defined by 
$$\Omega_p = \{\vec{\lambda}(p)\in \R^d \colon p \in \Sig\text{ is a periodic point}\}$$ is a subset of $\Omega_\A$. Theorem \ref{thm: simultaneous} can be used to establish the reverse inclusion:

\begin{thmx}\label{thm: spectrum} 
Let $\A \colon \Sig\to \glr$ be a typical cocycle. Then the closure $\ol{\Omega_p}$ of the Lyapunov spectrum over periodic points
coincides with the pointwise Lyapunov spectrum $\Omega_\A$.
\end{thmx}

Kalinin \cite{kalinin2011livvsic} showed that for any \hol continuous cocycle $\A \colon \Sig \to \glr$, the Lyapunov exponent $\vec{\lambda}(\mu)$ of any ergodic measure $\mu \in \E(\s)$ can be approximated by the Lyapunov exponents over periodic points. With an added assumption of typicality, Theorem \ref{thm: spectrum} shows that any $\vec{\a} \in \Omega_\A$ that is not necessarily given by the Lyapunov exponent $\vec{\lambda}(\mu)$ of some ergodic measure $\mu\in \E(\s)$ can also be approximated by the Lyapunov exponents over periodic points; see Subsection \ref{subsec: spectrum} for further discussions.

The paper is organize as follows. In Section \ref{sec: prelim} we introduce the setting and notations, and survey relevant preliminary results. In Section \ref{sec: proximal construction}, we assume 1-typicality and construct for any $x\in \N$ and $n\in \N$ a periodic point $q\in \Sig$ of period $n_q \in \N$ such that the orbit of $q$ shadows the forward orbit of $x$ upto time $n$ and that $\A^{n_q}(q)$ is proximal. In Section \ref{sec: simult proximality}, assuming typicality we extend the construction such that all exterior products of $\A^{n_q}(q)$ are proximal. In Section \ref{sec: proof other thms}, we show that such periodic orbit $q\in \Sig$ verifies the statement of Theorem \ref{thm: simultaneous} and derive from it the remaining theorems.

\subsection*{Acknowledgments}
The author would like to thank Rafael Potrie for suggesting the problem and Amie Wilkinson for helpful discussions.

\section{Preliminaries}\label{sec: prelim}
\subsection{Subshifts of finite type and standing notations}
An \textit{adjacency matrix} $T$ is a square matrix with entries in $\{0,1\}$.
A (two-sided) \textit{subshift of finite type} defined by a $q \times q$ adjacency matrix $T$ is a dynamical system $\s \colon \Sig \to \Sig$ where 
$$\Sig := \{(x_i)_{i \in \Z} \colon x_i \in \{1,2,\ldots,q\} \text{ and }T_{x_i,x_{i+1}} = 1 \text{ for all }i \in \Z\}$$
and $\s$ is the left shift operator on $\Sig$.

We will always assume that the adjacency matrix $T$ is \textit{primitive}, meaning that there exists $M\in \N$ such that all entries of $T^M$ are positive. The primitivity of $T$ is equivalent to the mixing property of the corresponding subshift of finite type $(\Sig,\s)$, and such constant $M$ is called the \textit{mixing rate} of $\Sig$.

We endow $\Sig$ with a metric $d$ defined as follows: for $x = (x_i)_{i \in \Z},y = (y_i)_{i \in \Z} \in \Sig$, we define
$$d(x,y)  := 2^{-k},$$ 
where $k$ is the largest integer such that $x_i = y_i$ for all $|i| < k$. Equipped with such a metric, the shift operator $\s$ becomes a hyperbolic homeomorphism on a compact metric space $\Sig$.

We define the \textit{local stable set} $\Wloc^s(x)$ \textit{of $x \in \Sig$} by 
$$\Wloc^s(x):=\{y \in \Sig \colon x_i = y_i \text{ for all } i \geq 0\}.$$
The \textit{stable set} $\W^s(x)$ \textit{of $x \in \Sig$} then defined by 
$$\W^s(x):=\{y \in \Sig \colon \s^ny \in \Wloc^s(\s^nx) \text{ for some } n\geq 0\}.$$ 
Analogously, the (local) stable set of $\s^{-1}$ is called the \textit{(local) unstable set $\W_{(\text{loc})}^u$ of $\s$}. Explicitly, they are defined as
$$\Wloc^u(x):=\{y \in \Sig \colon x_i = y_i \text{ for all } i\leq 0\}$$
and
$$\W^u(x):=\{y \in \Sig \colon \s^ny \in \Wloc^u(\s^nx) \text{ for some } n \leq 0\}.$$

A \textit{cylinder} containing $x = (x_i)_{i \in \Z}\in \Sig$ of length $n\in \N$ is defined by
$$[x]_n:=\{(y_i)_{i\in \Z} \in \Sig \colon  x_i = y_i \text{ for all }0 \leq i \leq n-1\}.$$ 
For any $x,y\in\Sig$ with $x_0 = y_0$ (ie, $y \in [x]_1$), we define
$$[x,y]:= \Wloc^u(x)\cap \Wloc^s(y).$$ In particular, $[x,y]$ is the unique point in $[x]_1$ whose forward orbit shadows that of $y$ and the backward orbit shadows that of $x$ synchronously.
 
We now introduce notations that will be assumed throughout the paper. 
\begin{notation}\label{not: 1} For any vectors $u,v,w \in \R^d$, we will denote their projectivization by $\u,\v,\w\in \P^{d-1}$, respectively. Similarly for hyperplanes $\U,\V,\WW \subseteq \R^d$, we will denote their projectivization by $\UU,\vv,\ww \subset \P^{d-1}$. Moreover, for linear maps appearing later such as $\A(x)$, $g$, $P$, $H^{s/u}$, $\psi$, $B$, and $R$, we will denote their induced actions on $\P^{d-1}$ by $\AA(x)$, $\gsf$, $\PP$, $\HH^{s/u}$, $\Psi$, $\BB$, and $\RR$, respectively. 
\end{notation}

Let $\rho$ be the angular metric on $\mathbb{P}^{d-1}$:
$$\rho(\u,\v) :=\min\{\mangle(u,v),\mangle(u,-v)\}.$$
Given any set $\mathsf{S} \subseteq \mathbb{P}^{d-1}$ and $\delta>0$, we denote the $\delta$-neighborhood of $\mathsf{S}$ by 
$$\CC(\mathsf{S},\delta):=\{\v\in \mathbb{P}^{d-1} \colon \rho(\v,\mathsf{S})\leq \delta\}.$$
For any $g\in \glr$, we define
$$\|\gsf\|_\rho :=\sup\limits_{\u \neq \v}\frac{\rho(\gsf\u,\gsf\v)}{\rho(\u,\v)}.$$
We call $\|\gsf\|_\rho$ the \emph{Lipschitz seminorm of $\gsf$ with respect to $\rho$} or simply \emph{$\rho$-norm of $\gsf$}. It is clear from the definition that $\|\cdot\|_\rho$ is sub-multiplicative under composition.
Straight from the definition again, we obtain the following identity which we will frequently use: 
\begin{equation}\label{eq: rho and cones}
\gsf(\CC(\v,\delta)) \subseteq \CC(\gsf \v, \delta \|\gsf\|_\rho).
\end{equation}

Denoting by $\|g\| = \a_1(g)$ the operator norm and $m(g):=\|g^{-1}\|^{-1} = \a_d(g)$ the conorm of $g$, we can bound the $\rho$-norm of $\gsf$ by the following expression involving the norm and conorm of $g$ and some uniform constant $C >0$ (see for instance, \cite[Lemma 3.7]{piraino2020weak} together  with \cite[Section 2.1.1]{duarte2016lyapunov}):
\begin{equation}\label{eq: rho and norm}
\|\gsf\|_{\rho} \leq C\cdot \|g\|\cdot m(g)^{-1}.
\end{equation}
This will later be used in bounding the $\rho$-norm of products of matrices drawn from a compact subset (ie, the image of $\A$) of $\glr$ .

\subsection{Fiber-bunched cocycles}
Let $\A \colon \Sig \to \glr$ be a continuous cocycle. Recall from the introduction that $\A^n(x):=\A(\s^{n-1}x)\ldots \A(x)$ denotes the product of $\A$ along the length $n$ orbit of $x$. Notice that such product $\A^n(x)$ satisfies the \emph{cocycle equation}: for any $x\in \Sig$ and $m,n\in \N$,
$$\A^{n+m}(x) = \A^n(\s^m x)\A^m(x).$$
By setting $\A^0 \equiv \id$ and $\A^{-n}(x) = \A^n(\s^{-n}x)^{-1}$ for $n\in \N$, the cocycle equation holds for any $m,n\in \Z$.

In this paper, we will focus on $\a$-\hol continuous cocycles with $\a \in (0,1]$ satisfying an additional assumption called the fiber-bunching.
An $\a$-\hol cocycle $\A\colon \Sig\to \glr$ is \emph{fiber-bunched} if 
$$\|\A(x)\|\cdot \|\A(x)^{-1}\| \cdot (1/2)^{\a} <1$$
for every $x\in \Sig$. 

If $\A$ is conformal, meaning that $\A(x)$ is conformal for every $x\in \Sig$, then it easily follows from the definition that $\A$ is fiber-bunched. Moreover, slight perturbations of conformal cocycles are also fiber-bunched. In fact, fiber-bunched cocycles may as well be thought of as nearly conformal cocycles.

One of the most important properties of fiber-bunched cocycles is that it guarantees the convergence of the canonical holonomies. The \emph{canonical local stable holonomy} is a family of matrices $H^{s}_{x,y}\in \glr$ defined for any $x,y\in \Sig$ with $y \in \Wloc^s(x)$, and it is explicitly given by 
$$H^s_{x,y} :=\lim\limits_{n\to \infty} \A^n(y)^{-1}\A^n(x).$$
The local stable holonomy satisfies the following properties:
\begin{enumerate}
\item $H^s_{x,x} = \id $ and $H^s_{y,z}\circ H^s_{x,y} = H^s_{x,z}$ for any $y,z \in \Wloc^s(x)$,
\item $\A(x) = H^s_{\s y,\s x} \circ \A(y) \circ H^s_{x,y}$,
\item $H^s\colon (x,y) \mapsto H^s_{x,y}$ is continuous as $x,y$ vary continuously while satisfying the relation $y \in \Wloc^s(x)$. 
\end{enumerate}
The \textit{local unstable holonomy} $H^u_{x,y}$ is likewise defined as
$$H^u_{x,y} :=\lim\limits_{n\to -\infty} \A^n(y)^{-1}\A^n(x)$$
for any $y\in \Wloc^u(x)$, and it satisfies analogous properties listed above with $s$ and $\s$ replaced by $u$ and $\s^{-1}$, respectively.

Recalling that $\a\in (0,1]$ is the \hol exponent of $\A$, the canonical holonomies are $\a$-\hol continuous (see \cite{kalinin2013cocycles}): there exists $C>0$ such that for any $y \in \Wloc^{s/u}(x)$,
$$
\|H^{s/u}_{x,y}-\id\| \leq C\cdot d(x,y)^\a.
$$
It then follows that fiber-bunced cocycles have the \emph{bounded distortion property}: there exists $C \geq 1$ such that for any $n \in \N$ and $y \in [x]_n$, we have
$$
C^{-1} \leq \frac{\|\A^n(x)\|}{\|\A^n(y)\|} \leq C.
$$
Indeed, by setting $z:=[x,y]$, we have
\begin{equation}\label{eq: same cylinder}
\A^n(x) = H^u_{\s^n z, \s^n x}H^s_{\s^n y, \s^n z} \A^n(y) H^s_{z,y}H^u_{x,z}
\end{equation}
and the bounded distortion property follows from the \hol continuity of the canonical holonomies.

We conclude this subsection by introducing a notation that will prove useful later.
Throughout the paper, we will have a distinguished periodic point $p\in \Sig$ from the typicality assumption. For any $q \in [p]_1$, we define the \textit{rectangle through $p$ and $q$} as
\begin{equation}\label{eq: R}
R_q := H^u_{[p,q],p}\circ H^s_{q,[p,q]}\circ H^u_{[q,p],q} \circ H^s_{p,[q,p]}.
\end{equation}

Since the canonical holonomies are \hol continuous, $R_q$ uniformly approaches the identity if the length of \textit{either} pair of opposite edges of the rectangle approaches 0. In fact, assuming that the edge connecting $[p,q]$ and $p$ is the shorter leg of the rectangle, it can be shown that there exist $C,\theta>0$ depending only on $\Sig$ and $\A$ such that
\begin{equation}\label{eq: R_q}
\|R_q - \id\| \leq C\cdot d([p,q],p)^\theta.
\end{equation}
See Bochi-Garibaldi \cite{bochi2019extremal}. This \hol continuity of $R_q$ has the following consequence: 
\begin{lem}\label{lem: rectangle}
For any $\d>0$, there exists $m=m(\d)\in \N$ such that
$$
\min \{d([p,q],p),d([p,q],q)\} \leq 2^{-m} \implies \RR_q^{\pm 1}\CC(\mathsf{S},c)\cup\CC(\RR_q^{\pm 1}\mathsf{S},c) \subseteq \CC(\mathsf{S},c+\delta)
$$
for any $\mathsf{S}\subseteq \mathbb{P}^{d-1}$ and $c>0$.
\end{lem}
\begin{proof}
Suppose $\v\in \RR_q^{\pm 1}\CC(\mathsf{S},c)$, meaning that there exists $\w \in \mathsf{S}$ with $\rho(\RR_q^{\mp 1}\v,\w) \leq c$. Given $\delta>0$, we can find $m = m(\delta)$ using \eqref{eq: R_q} such that $\rho(\v, \RR_q^{\mp 1}\v) < \delta$ when $d([p,q],p)<2^{-m}$ or $d([p,q],q)<2^{-m}$.
Then the triangle inequality $\rho(\v,\w) \leq \rho(\v,\RR_q^{\mp 1}\v)+\rho(\RR_q^{\mp 1}\v,\w)$ gives the inclusion $\RR_q^{\pm 1}\CC(\mathsf{S},c) \subseteq \CC(\mathsf{S},c+\delta)$. The other inclusion $\CC(\RR_q^{\pm 1}\mathsf{S},c)\subseteq \CC(\mathsf{S},c+\delta)$ can likewise be established.
\end{proof}

\subsection{Typical cocycles}
We now describe the typicality assumption that appeared in main results from the introduction.

Let $p\in \Sig$ be a periodic point. We say $z \in \Sig$ is a \emph{homoclinic point} of $p$ if $z$ belongs to $\W^s(p) \cap \W^u(p) \setminus \{p\}$; that is, it is a point whose orbit synchronously approaches the orbit of $p$ both forward and backward time. 
We denote the set of all homoclinic points of $p$ by $\mathcal{H}(p)$. For uniformly hyperbolic systems such as $(\Sig,\s)$, $\mathcal{H}(p)$ is dense for every periodic point $p$. For any $z \in \mathcal{H}(p)$, we define its \emph{holonomy loop} by $$\psi_z:=H^{s}_{z,p}\circ H^u_{p,z}.$$
\begin{defn}[Typicality]\label{defn: typical}
We say a fiber-bunched cocycle $\A$ is \emph{1-typical} if there exist a periodic point $p \in \Sig$ and a homoclinic point $z\in \mathcal{H}(p)$ (we say such $(p,z)$ is the \emph{typical pair} of $\A$) such that
\begin{enumerate}
\item $P:=\A^{\text{per}(p)}(p)$ has simple eigenvalues of distinct moduli, and
\item Denoting the eigenvectors of $P$ by $\{v_1,\ldots,v_d\}$, for any $I,J \in \{1,\ldots,d\}$ with $|I|+|J| \leq d$,  the set of vectors
$$\{\psi_z(v_i) \colon i \in I\} \cup \{v_j \colon j \in J\}$$
is linearly independent. 
\end{enumerate}

If $\A^{\wedge t}$ is 1-typical with respect to the \textit{same} typical pair $(p,z)$ for all $1 \leq t \leq d-1$, then we say $\A$ is \emph{typical}. 
\end{defn}

Two conditions from the 1-typicality assumption above are called \emph{pinching} and \emph{twisting}, respectively. We note that the typicality assumption first appeared in Bonatti and Viana \cite{bonatti2004lyapunov} in a slightly weaker form than our definition above; they only require 1-typicality of $\A^{\wedge t}$ for $1 \leq t\leq d/2$, and they do not ask the typical pair $(p,z)$ to be the same pair over different $t$ (ie, the typical pair could differ for each $\A^{\wedge t}$). With such version of typicality, they proved that typical cocycles have simple Lyapunov exponents with respect to any ergodic measures with continuous local product structure. They also showed in the same paper that the set of typical cocycles forms an open and dense subset among the set of fiber-bunched cocycles. Although our version of typicality is slightly stronger, their perturbation technique still applies to show that the set of typical cocycles is open and dense.

We note that the only role of the fiber-bunching condition on $\A$ from the above definition is to guarantee the convergence of the canonical holonomies. In fact, we can speak of the 1-typicality assumption for cocycles that are not necessarily fiber-bunched but still admit uniformly continuous holonomies. For instance, while the exterior product cocycles $\A^{\wedge t}$ may not necessarily be fiber-bunched, they still admit \hol continuous canonical holonomies, namely $(H^{s/u})^{\wedge t}$. So we may still consider the 1- typicality assumption on the exterior product cocycles as did in Definition \ref{defn: typical}.

We also note that the typicality assumption has appeared elsewhere in the literature for different purposes. For instance, the author \cite{park2020quasi} showed that the singular value potentials $\Phi_\A^s$ of typical cocycles $\A$ have unique equilibrium states; see Section \ref{sec: proof other thms} for details.

For simplicity, we will always assume that $p$ is a fixed point by passing to the power $\A^{\text{per}(p)}$ if necessary (because powers of typical cocycles are typical). Also for any homoclinic point $z\in \mathcal{H}(p)$, $\s^nz$ is also a homoclinic point of $p$  for any $n \in \Z$. Moreover, their holonomy loops are conjugated by $P^n$:
$$P^n\psi_z = \psi_{\s^{n}z} P^n.$$ 
This implies that if $z\in \mathcal{H}(p)$ satisfies the twisting condition, then so does any point $\s^n z\in \mathcal{H}(p)$ in its orbit. Hence, we may replace $z$ by any point in its orbit without destroying the twisting assumption. This observation will be used later in the paper.

\subsection{Spannability and transversality}\label{subsec: spannable and transverse}
We now describe the notion of spannability introduced in Bochi and Garibaldi \cite{bochi2019extremal}.
Consider any $x,y\in \Sig$, $x' \in \Wloc^u(x)$, and $y'\in \Wloc^s(y)$ such that $y'=\s^{n}x'$ for some $n\in \N$. We say such points form a \textit{path (of length $n$) from $x$ to $y$}; that is, we say that $$x \xrightarrow{\Wloc^u(x)}x' \xrightarrow{\s^n}y'\xrightarrow{\Wloc^s(y)}y$$ forms a path from $x$ to $y$. Given any such path together with a fiber-bunched cocycle $\A \colon \Sig \to \glr$, by the \emph{cocycle over the path} we mean the expression
$$B_{x,y}:=H^s_{y',y}\A^{n}(x')H^u_{x,x'}.$$ 
When the context is clear, we will simply use the terminology ``path from $x$ to $y$'' to also refer to $B_{x,y}$.

\begin{defn} A fiber-bunched cocycle $\A \colon \Sig \to \glr$ is \textit{spannable} if for every $x,y \in \Sig$ and $v \in \R^d \setminus \{0\}$, there exist $d$ paths $B_1,\ldots, B_d$ from $x$ to $y$ such that $\{B_1(v),\ldots,B_d(v)\}$ forms a basis of $\R^d$.
\end{defn}

Bochi and Garibaldi \cite{bochi2019extremal} showed that under a suitable irreducibility assumption, strongly-bunched (an assumption stronger than the fiber-bunching assumption) cocycles are spannable. Using the compactness of $\Sig$, they also showed that any spannable cocycles $\A \colon \Sig \to \glr$ are \emph{uniformly spannable}, meaning that the length of paths $B_1,\ldots,B_d$ can be uniformly bounded above and the angle between any two vectors in $\{B_1(v),\ldots,B_d(v)\}$ can be uniformly bounded below.

\begin{defn}
A  fiber-bunched cocycle $\A$ is \textit{transverse} if for any $x,y\in \Sig$, a nonzero vector $v\in \R^d$, and a hyperplane $W\subset \R^d$, there exists a path $B_{x,y}$ such that $B_{x,y}v \not\in W$. 
\end{defn} 

Again, using the compactness of $\Sig$ and $\P^{d-1}$, it follows that a transverse fiber-bunched cocycle is \textit{uniformly transverse}, meaning that there exist $\ep>0$ and $N\in\N$ such that the path $B_{x,y}$ can be chosen to have its length at most $N$ and that 
$$\BB_{x,y}\v \not\in\CC(\ww,\ep).$$

We note that a cocycle $\A$ is transverse if and only if spannable. In fact, the ``if'' direction is obvious, and the ``only if'' direction follows from an inductive argument. Typical cocycles are necessarily irreducible, and hence, strongly-bunched cocycles satisfying the typicality assumption are transverse by \cite{bochi2019extremal}. 

A version of transversality (Theorem \ref{thm: simult transverse}) we will prove for typical cocycles is more general and explicit. We consider fiber-bunched cocycles that are not necessarily strongly bunched. Moreover, using concrete data (pinching and twisting) of the typicality assumption, we will construct a common path simultaneously satisfying the transversality condition for all exterior product cocycles $\A^{\wedge t}$, $t \in \{1,\ldots,d-1\}$. Such simultaneous transversality will later be used in proving Theorem \ref{thm: simultaneous}. 

\subsection{Proximal maps}

\begin{defn}\cite{benoist1996actions}\label{defn: ep-proximal} We say $g \in  \glr$ is \textit{proximal} if it has a unique eigenvalue of maximal modulus and this eigenvalue has algebraic multiplicity one. For a proximal map $g \in \glr$, we set $v_g\in \R^d$ be the unit length eigenvector of $g$ corresponding to the eigenvalue of maximum modulus and $V^<_g \subset \R^d$ be the $g$-invariant hyperplane complementary to $v_g$.

For $\ep>0$ and a proximal map $g\in \glr$, we say $g$ is \emph{$\ep$-proximal} if $\rho(\v_g,\vv^<_g) \geq 2\ep$ and 
\begin{equation}\label{eq: ep proximal}
\gsf(\P^{d-1} \setminus \CC(\vv^<_g,\ep)) \subseteq \CC(\v_g,\ep) \text{ and }\|\gsf \mid_{ \P^{d-1} \setminus \CC(\vv^<_g,\ep)}\|_\rho \leq \ep.
\end{equation}
\end{defn}

When $\ep$ is sufficiently small, $\ep$-proximal maps behave similar to rank 1 projections. Recalling from \eqref{eq: singular and eigen} that $\mu_1(g) := \log \|g\|$ is the logarithm of the norm and $\chi_1(g) := \log|\text{eig}_1(g)|$ is the logarithm of the spectral radius, one way to make use of this intuition is via the following proposition:

\begin{prop}\cite{benoist1996actions}\label{prop: property of proximal map}
Given any $\ep>0$, there exists $C=C(\ep)>0$ such that any $\ep$-proximal linear transformation $g \in \glr$ satisfy
$$\mu_1(g)-C\leq \chi_1(g)\leq  \mu_1(g).$$ 
\end{prop}

When we later construct the required periodic point $q$ to establish Theorem \ref{thm: simultaneous} (see also Theorem \ref{thm: model thm}), we will construct it so that $\A^{n_q}(q)$ and its exterior products are $\ep$-proximal and then apply the above proposition. For the construction, we will have to control the $\rho$-norm of $\mathsf{A}^{n_q}(q)$ using the following proposition of Abels, Margulis, and Soifer.

\begin{prop}\cite[Proposition 4.2]{abels1995semigroups}\label{prop: AMS} We can associate with every $g\in\glr$ a hyperplane $\U_{g} \subset \R^d$ with the following property. For every $\ep>0$ there is a constant $c = c(\ep)>0$ such that for every $g\in \glr$, 
$$\|\gsf\mid_{\mathbb{P}^{d-1} \setminus \CC(\UU_{g},\ep)}\|_\rho \leq c.$$
\end{prop}

The hyperplane $\U_{g}$ admits a simple description via the $KAK$-decomposition; if $g = U\Lambda V^*$ is a $KAK$-decomposition of $g$ where $\Lambda$ is a diagonal matrix whose entries are the singular values of $g$ listed in a decreasing order, then $\U_{g}$ may be taken to be $V\langle e_2,\ldots, e_d\rangle$. 

The control on the $\rho$-norm of $\mathsf{A}^{n_q}(q)$ from Proposition \ref{prop: AMS} will then later be used to show that $\A^{n_q}(q)$ is $\ep$-proximal via the following criteria of Tits:
\begin{prop}\label{prop: tits} Suppose there exists a compact subset $\mathsf{K}\subseteq \P^{d-1}$ such that 
$$
\gsf(\mathsf{K}) \subset \mathsf{K}^{\mathrm{o}} \text{ and } \|\gsf \mid_\mathsf{K}\|_\rho<1.
$$
Then $g$ is proximal, and $\gsf$ has a unique fixed point $\v_g\in \mathsf{K}^{\mathrm{o}}$ and $\vv_g^<$ does not intersect $\mathsf{K}$.

Moreover, given any $\ep>0$, there exists $\xi = \xi(\ep)>0$ such that if $\mathsf{K}$ is given by $\CC(\w,3\ep)$ for some $\w\in \P^{d-1}$ and 
$$
\gsf(\CC(\w,3\ep)) \subseteq \CC(\w,\ep) \text{ and } \|\gsf \mid_{ \CC(\w,3\ep)}\|_\rho<\xi,
$$
then $g$ is $\ep$-proximal. 
\end{prop}
\begin{proof}
For the first statement, see \cite{tits1972free}. 
For the second statement, since $\v_g \in \CC(\w,\ep)$ and $\vv^<_g$ does not intersect $\CC(\w,3\ep)$, we have $\rho(\v_g,\vv^<_g) \geq 2\ep$ which verifies the first condition of the $\ep$-proximality. The remaining two conditions \eqref{eq: ep proximal} can be met by choosing $\xi$ sufficiently small depending on $\ep$.
\end{proof}

\section{Construction of proximal maps from 1-typicality}\label{sec: proximal construction}
The goal of this section is prove the following theorem for 1-typical cocycles. We will later need a generalized version of this result (Theorem \ref{thm: construct proximal general}) in order to prove Theorem \ref{thm: simultaneous}. 
Since the proof for the generalized result is essentially identical, we will first focus on the simplified result and illustrate the construction in full detail. 

\begin{thm}\label{thm: model thm}
Let $\A \colon \Sig \to \glr$ be a 1-typical cocycle. Then there exists $\tau_0 = \tau_0(\A)>0$ such that for any $\tau \in (0,\tau_0)$, there exists $k = k(\tau) \in \N$ with the following property: for any $x \in \Sig$ and $n \in \N$ there exists a periodic point $q \in\Sig$ of period $n_q \in [n,n+k]$ such that 
\begin{enumerate}
\item $\A^{n_q}(q)$ is $\tau$-proximal, and
\item there exists $j\in \N$ such that $\s^j q \in [x]_n$.
\end{enumerate}
\end{thm}

We make a few simplifications and set up the relevant notations and lemmas before we begin the proof. We will continue to assume the notations set in Notation \ref{not: 1}.

First, we make a few simplifying assumptions. First, we continue to assume that $p$ is a fixed point and denote the eigenvectors of $P:=\A(p)$ by $\{{v_1},\ldots,v_{d}\}$, listed in the order of decreasing absolute values for their corresponding eigenvalues. We also define the $P$-invariant hyperplanes $$\WW_i :=\text{span}\{ v_1,\ldots, v_{i-1},v_{i+1},\ldots,v_{d}\}\subseteq \R^d .$$

Consider now the distinguished homoclinic point $z \in \mathcal{H}(p)$ of $p$ from the typicality assumption. Since any point $\s^n z$ in the orbit of $z$ satisfies the twisting assumption, by replacing $z$ by a suitable preimage in its backward orbit, we may suppose that $z$ lies in $\Wloc^u(p)$. For any $\ell \in \N$ such that $\s^\ell z\in \Wloc^s(p)$, the $\A$-equivariance property of the canonical holonomies implies that the following relationship holds for the holonomy loop $\psi_z$: 

\begin{equation}\label{eq: psi ell}
P^\ell \psi_z = H^{s}_{\s^\ell z, p}\A^\ell(z)H^{u}_{p,z}.
\end{equation}

We now introduce a few preliminary lemmas. The first lemma, whose proof is clear, exploits the pinching assumption on $P$.

\begin{lem}\label{lem: powers of P}
Suppose $P \in \glr$ has simple eigenvalues of distinct moduli with corresponding eigenvectors $\{v_1,\ldots,v_d\}$. Then for any $\ep>0$ there exists $N = N(\ep)\in \N$ such that for any $\v\in \P^{d-1}$ there exists $a \in [0,N]$ satisfying
$$\PP^a \v \in \bigcup\limits_{i=1}^d\CC(\v_i,\ep)$$
\end{lem}

\begin{lem}\label{lem: connect}
Let $x,y,z\in \Sig$, and consider paths $B_1:=H^s_{y_0,y}\A^n(x_0) H^u_{x,x_0}$ from $x$ to $y$ (via $x_0$ and $y_0:=\s^n x_0$) and $B_2:= H^s_{z_1,z}\A^m(y_1)H^u_{y,y_1}$ from $y$ to $z$ (via $y_1$ and $z_1:=\s^m y_1$). Setting $w:=\s^{-n}[y_0,y_1] \in \Wloc^u(x)$ and $\wt{w}:=\s^{m+n}(w)\in \Wloc^s(z)$, the following identity holds:
$$H^s_{\wt{w},z}\A^{n+m}(w)H^u_{x,w}= B_2 R_{\s^n w}B_1$$
where $R_{\s^n w}:=H^u_{y_1,y}H^s_{\s^n w,y_1}H^u_{y_0,\s^n w}H^s_{y,y_0}$ is defined as in \eqref{eq: R} through opposite vertices $y$ and $\s^nw$.
\end{lem}
\begin{proof}
Using the properties of the canonical holonomies, we have
$$
H^u_{y_0,\s^n w}H^s_{y,y_0}B_1 = H^u_{y_0,\s^n w}\A^n(x_0)H^u_{x,x_0} = \A^n(w)H^u_{x,w}.
$$
Likewise, we have $B_2H^u_{y_1,y}H^s_{\s^n w,y_1} = H^s_{\wt{w},z}\A^m(\s^nw)$, and putting these two equations together completes the proof.
\end{proof}

\subsection{Transversality from 1-typicality}
Bochi and Garibaldi \cite{bochi2019extremal} showed that typical cocycles are spannable, and hence, uniformly spannable. It then follows that such cocycles are uniformly transverse. We, however, will prove the same result using a different method that makes use of the typicality assumption more directly.

\begin{thm}\label{thm: 1-transversality}
Let $\A \colon \Sig \to \glr$ be a 1-typical cocycle. Then $\A$ is uniformly transverse: there exist $N_1\in \N$ and $\ep_1>0$ such that for any $x,y\in \Sig$, a nonzero vector $v\in \R^d$, and a hyperplane $V\subseteq \R^d$, there exists a path $B:=B_{x,y}$ of length at most $N_1$ satisfying $$ \BB\v \not\in\CC(\vv,\ep_1).$$ 
\end{thm}

Prior to beginning the proof, we fix a few constants. First, the twisting assumption on the holonomy loop $\psi:=\psi_z$ ensures that all coefficients $c_{i,j}$ from the expansion
$$\psi v_i  = \sum\limits_{j=1}^d c_{i,j}v_j$$
are nonzero for any $1\leq i,j \leq d$. In particular, we can choose $\d>0$ depending only on $\A$ such that $\Psi:=\mathbb{P}\psi$ satisfies 
\begin{equation}\label{eq: delta}
\Psi\Big(\bigcup\limits_{i=1}^d \CC(\v_i,\delta)\Big) \subseteq \Big(\bigcup\limits_{i=1}^d\CC(\ww_i,\delta)\Big)^c
\end{equation}
where $\ww_i := \P(\mathbb{W}_i)$.
Note from the pinching assumption on $P$, applying powers of $\PP$ maps $\Big(\bigcup\limits_{i=1}^d\CC(\ww_i,\delta)\Big)^c$ close to $\v_1$. In particular, for any $\ep>0$ we can choose $\ell(\ep) \in \N$ large enough such that 	
\begin{equation}\label{eq: 1}
\PP^\ell \Psi\Big( \bigcup\limits_{i=1}^d \CC(\v_i,\delta)\Big) \subseteq\CC(\v_1,\ep)
\end{equation}
for any $\ell \geq \ell(\ep)$.

In the proof of Theorem \ref{thm: 1-transversality}, we will also make use of the inverse cocycle. The \textit{inverse cocycle} $\A^{-1}$ is a cocycle over $(\Sig,\s^{-1})$ defined by 
$$\A^{-1}(x) := \A(\s^{-1}x)^{-1}.$$
In particular, the iteration of $\A^{-1}$ is then given by $\A^{-n}(x) = \A^n(\s^{-n}x)^{-1}$. 

\begin{rem}\label{rem: inverse}
Many properties of $\A$ are inherited to the inverse cocycles $\A^{-1}$. For instance, $\A^{-1}$ is fiber-bunched if and only if $\A$ is. Moreover, denoting the canonical holonomies of $\A^{-1}$ by $H^{s/u,-}$, we have $H^{s/u,-}_{x,y} = H^{u/s}_{x,y}$ from their definitions. Noticing also that the inverse of $H^{s/u}_{x,y}$ is $H^{s/u}_{y,x}$, we have 
$$\psi_z^{-} = H^{s,-}_{z,p}\circ H^{u,-}_{p,z} = H^u_{z,p} \circ H^s_{p,z}  = (H^s_{z,p} \circ H^u_{p,z})^{-1}=(\psi_z)^{-1}.$$

It then follows that $\A^{-1}$ is typical if and only if $\A$ is. This is because the pinching assumption holds for $P$ if and only if it holds for $P^{-1}$. Moreover, given any two index sets $I,J \in \{1,\ldots,d\}$ with $|I|+|J| \leq d$,  the set of vectors
$$\{\psi_z^-(v_i) \colon i \in I\} \cup \{v_j \colon j \in J\}$$ is linearly independent if and only if
$$\{v_i \colon i \in I\} \cup \{\psi_z(v_j) \colon j \in J\}$$ is linearly independent. But then the latter set is linearly independent if $\A$ is twisting. This shows that $\A$ is typical if and only if $\A^{-1}$ is.
\end{rem}

\begin{proof}[Proof of Theorem \ref{thm: 1-transversality}]
The proof resembles that of \cite[Theorem 4.1]{park2020quasi} closely. Throughout the proof, refer to Figure \ref{fig1} for a visual representation of the proof.
Let $x,y \in \Sig$, $v\in \R^d$, and $V \subset \R^d$ be given.
We begin by building a path from $x$ to $p$ with some extra properties:

\begin{lem}\label{lem: 1}
 For any $\ep>0$ and $k\in \N$, there exists $N \in \N$ with the following property: for any $x\in \Sig$ and $v\in \R^d$, there exists a path $B_{x,p}:=H^s_{x_1,p}\A^{n(x)}H^u_{x,x_0}$ where $x_1:=\s^{n(x)}x_0$ such that its length $n(x)$ is at most $N$, the distance $d(x_1,p)$ is at most $2^{-k}$, and that 
$$\rho(\BB_{x,p}\v,\v_1) <\ep.$$
\end{lem}
\begin{proof}
Let $\ep>0, k\in \N,  x\in \Sig$ and $v\in \R^d$ be given.
From the mixing property of the subshift $(\Sig,\s)$, there exists $M\in \N$ such that we can find $w_0 \in \Wloc^u(x)$ and $w_1:=\s^{M}w_0 \in \Wloc^s(p)$. Setting $m:=m(\d/2)$ from Lemma \ref{lem: rectangle},  by increasing $M$ if necessary, we may assume that $d(w_1,p) \leq 2^{-m}$ by replacing $w_1$ by $\s^{m}w_1$.  Since $\d$ depends only on the cocycle,  the (possibly increased) constant $M$ also depends only on the cocycle.

Applying Lemma \ref{lem: powers of P} to $\u:=\HH^s_{w_1,p}\AA^{M}(w_0)\HH^u_{x,w_0}\v$ gives $a \in [0,N']$ where $N':=N(\d/2)$ such that 
\begin{equation}\label{eq: v control}
\PP^a\u=H^s_{\s^aw_1,p}\AA^{M+a}(w_0)\HH^u_{x,w_0}\v \in \CC(\v_i,\d/2)
\end{equation}
 for some $1 \leq i \leq d$. 

Fix $\ell \in \N$ such that $\ell \geq \ell(\ep)$ where $\ell(\ep)$ is from \eqref{eq: 1} and that $d(\s^{\ell} z,p) \leq 2^{-k}$.
Now considering $\wt{x}_0:=[\s^a w_1,z]$ and $x_0:=\s^{-M-a}\wt{x}_0 \in \Wloc^u(x)$, Lemma \ref{lem: connect} allows us to connect two paths $H^s_{\s^aw_1,p}\A^{M+a}(w_0)H^u_{x,w_0}$ and $H^s_{\s^\ell z,p}\A^\ell(z)H^u_{p,z}$: setting $n(x):=M+a+\ell$ and $x_1:=\s^{n(x)}x_0$, we have
\begin{align*}
B_{x,p}:=H^s_{x_1,p}\A^{n(x)}(x_0)H^u_{x,x_0}&=\Big(H^s_{\s^\ell z,p}\A^\ell(z)H^u_{p,z}\Big)R_{\wt{x}_0}\Big(H^s_{\s^aw_1,p}\A^{M+a}(w_0)H^u_{x,w_0}\Big)\\
&=P^\ell \psi R_{\wt{x}_0}H^s_{\s^aw_1,p}\A^{M+a}(w_0)H^u_{x,w_0},
\end{align*}
where $R_{\wt{x}_0}$ is defined as in \eqref{eq: R}. Note that we used the identity \eqref{eq: psi ell} in the second equality. 
Denoting this new path by $B_{x,p}$, its length $n(x)$ is bounded above by $N:=M+N'+\ell$, and we have $\BB_{x,p}\v = \PP^\ell \Psi \RR_{\wt{x}_0}\PP^a \u$ from the definition of $\u$.

\begin{figure}[H]
\caption{}
\includegraphics[width=14cm]{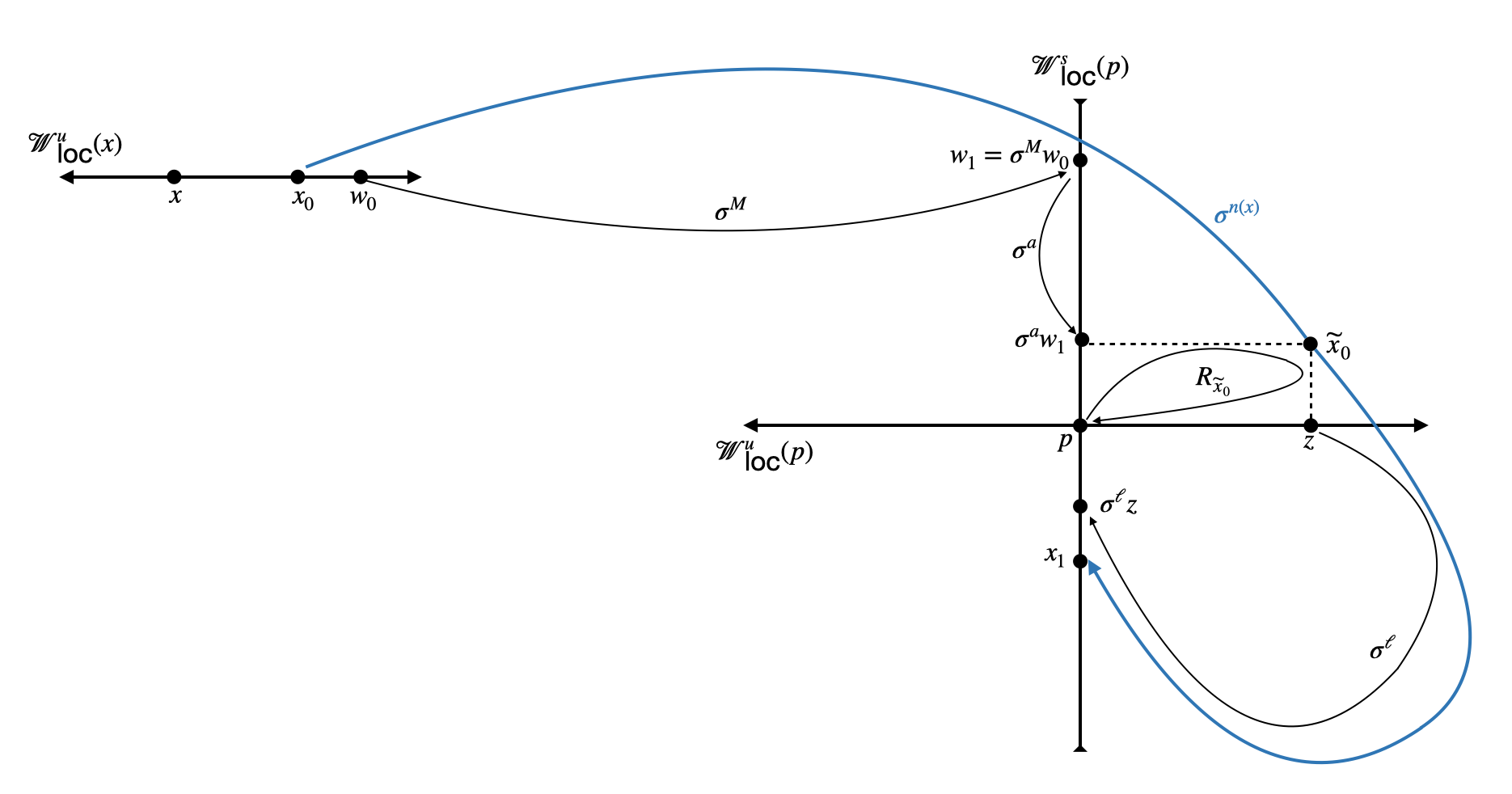}
\label{fig1}
\centering
\end{figure}

Note that one of the legs of $R_{\wt{x}_0}$ is the edge connecting $\s^aw_1$ and $p$, and its length is at most $2^{-m}$ from the assumption that $d(w_1,p) \leq 2^{-m}$. The choice of $m = m(\d/2)$ from Lemma \ref{lem: rectangle} then ensures that $\RR_{\wt{x}_0}$ does not move any direction off itself by more than $\d/2$ in angle. 
 Since we have $\PP^a\u \in \CC(\v_i,\d/2)$ for some $i$ from \eqref{eq: v control}, this implies that $\RR_{\wt{x}_0} \PP^a\u \in \CC(\v_i,\d)$. Then the choice \eqref{eq: 1} of $\ell$ ensures that
$$\BB_{x,p}\v = \PP^\ell \Psi \RR_{\wt{x}_0}\PP^a \u \in \CC(\v_1,\ep).$$
Lastly, noting that $d(x_1,p) =d(\s^\ell z,p) \leq 2^{-k}$, this completes the proof.
\end{proof}

Recalling that $\mathbb{W}_d = \text{span}(v_1,\ldots, v_{d-1})$ and that $v_1\wedge \ldots \wedge v_{d-1}$ is the attracting fixed point of $P^{\wedge d-1}$, the above lemma can also be applied to hyperplanes:
\begin{lem}\label{lem: 2}
 For any $\ep>0$ and $k\in \N$, there exists $N \in \N$ with the following property: for any point $x\in \Sig$ and hyperplane $V\subseteq \R^d$, there exists a path $B_{x,p}:=H^s_{x_1,p}\A^{n(x)}H^u_{x,x_0}$ where $x_1:=\s^{n(x)}x_0$ such that its length $n(x)$ is at most $N$, the distance $d(x_1,p)$ is at most $2^{-k}$, and that 
$$\rho(\BB_{x,p}\mathsf{V},\ww_d) <\ep.$$
\end{lem}
\begin{proof}
We will explain here how the arguments from Lemma \ref{lem: 1} above can also be applied to the exterior product cocycle $\A^{\wedge d-1}$. First, it is clear that $P^{\wedge d-1}$ has simple eigenvalues of distinct moduli; its eigenvectors are $\{u_i\}_{1 \leq i \leq d}$ where $u_i:=v_1 \wedge \ldots \wedge v_{i-1} \wedge v_{i+1} \wedge \ldots \wedge v_{d}$. This verifies the pinching assumption for $\A^{\wedge d-1}$.

As for the twisting assumption, notice from the proof of Lemma \ref{lem: 1} that the only use of the twisting assumption was \eqref{eq: delta}, which ensured that vectors sufficiently close to the eigenvectors $v_i$ of $P$ can be twisted by $\psi$ away from each hyperplane $W_j$. 

The same can be shown to be true for the eigenvectors $\{u_i\}_{1\leq i \leq d}$. Indeed, if we write 
$$\psi^{\wedge d-1}u_i= \sum\limits_{j=1}^d c_{i,j} u_j,$$
 then we can find $c_{i,j}$ by taking the exterior product with $v_j$ on each side. Since $v_j \wedge \psi^{\wedge d-1}u_i=
v_j \wedge (\psi v_1 \wedge \ldots  \wedge \psi v_{i-1} \wedge\psi v_{i+1} \wedge \ldots \wedge \psi v_{d})$ is non-zero for each $i$ and $j$ from the twisting assumption (see Definition \ref{defn: typical}), it follows that $c_{i,j}$ is nonzero for all $i$ and $j$. This means that $\psi^{\wedge d-1}$ satisfies the twisting condition analogous to \eqref{eq: delta} for the eigenvectors $\{u_i\}_{1\leq i \leq d}$ of $P^{\wedge d-1}$. Hence, the same argument as in Lemma \ref{lem: 1} applies when a vector $v$ is replaced by a hyperplane $V$.
\end{proof}

Continuing on with the proof of Theorem \ref{thm: 1-transversality},  consider the angle $\eta:=\rho(\v_1,\ww_1) >0$. By setting $\ep := \eta/4$ and $k := m(\eta/4)$ and applying Lemma \ref{lem: 1}, we obtain $N \in \N$ and a path 
$$B_{1}:= H^s_{x_1,p}\A^{n(x)}(x_0)H^u_{x,x_0}$$
 of length at most $N$ such that $d(x_1,p) \leq 2^{-m(\eta/4)}$ and that $$\rho(\BB_1\v,\v_1) <\eta/4.$$

We now apply Lemma \ref{lem: 2} to the inverse cocycle $\A^{-1}$ with the same constants $\ep,k$ to obtain a path $B_2^{-}$ from $y$ to $p$ such that 
$$\rho(\BB_2^-\vv,\ww_1)<\eta/4;$$
 note that we have used the fact that $v_2\wedge \ldots \wedge v_d$ is the attracting fixed point of $(P^{\wedge d-1})^{-1}$ and that $\WW_1 = \text{span}(v_2,\ldots,v_d)$. 
By increasing $N$ if necessary (to the constant $N$ obtained from Lemma \ref{lem: 2}), we may assume that the length of $B_2^{-}$ is also bounded above by $N$. Since the local stable holonomies of $\A^{-1}$ coincide with the local unstable holonomies of $\A$, the path $B_2^{-}$ may be written as $y \xrightarrow{H^{s}} y_0 \xrightarrow{\s^{-n(y)}} y_1 \xrightarrow{H^{u}} p$ for some $n(y) \leq N$. Here $H^{s/u}$ are the canonical holonomies of the cocycle $\A$.

We now consider a path $B_2$ from $p$ to $y$ defined as the inverse of $B_2^-$. In particular, we have
$$B_2:=H^s_{y_0,y}\A^{n(y)}(y_1)H^u_{p,y_1}$$ where $y_1 \in \Wloc^u(p)$ with $d(y_1,p) \leq 2^{-m(\eta/4)}$ and $y_0:=\s^{n(y)}y_1 \in \Wloc^s(y)$. 

By setting $r:=\s^{-n(x)}[x_1,y_1]$ and $\wt{r}:=\s^{n(x)+n(y)}(r)$, Lemma \ref{lem: connect} allows us to create a new path $B$ by joining $B_1$ and $B_2$:
$$B:=H^s_{\wt{r},y}\A^{n(x)+n(y)}(r)H^u_{x,r}=B_2 R_{\s^{n(x)}r}B_1.$$
Note that the length of $B$ is uniformly bounded above by $2N$.
Since $d(x_1,p) \leq 2^{-m(\eta/4)}$, Lemma \ref{lem: rectangle} ensures that $\RR_{\s^{n(x)}r}$ does not move any direction off itself by more than $\eta/4$ in angle, and we have $$\rho(\RR_{\s^{n(x)}r}\BB_1\v,\v_1) \leq \rho(\BB_1\v,\v_1)+\eta/4 \leq \eta/2.$$ In particular, we have 
\begin{align*}
\rho(\BB_{2}^{-1}\vv,\RR_{\s^{n(x)}r}\BB_1\v) &\geq \rho(\v_1,\ww_1)-\rho(\v_1,\RR_{\s^{n(x)}r}\BB_1\v)-\rho(\BB_2^{-1}\vv,\ww_1) \\
& \geq \eta - \eta/2-\eta/4 = \eta/4.
\end{align*}
Since the length of $B_2$ is uniformly bounded above by $N$, we can choose $\ep_1>0$ such that
$\rho(\BB\v,\vv) \geq \ep_1.$ By setting $N_1:=2N$, we complete the proof of Theorem \ref{thm: 1-transversality}.
\end{proof}

\begin{rem}
In \cite[Theorem 4.1]{park2020quasi}, a similar construction appears to establish quasi-multiplicativity of typical cocycles. Compared to our proof above, the difference is that instead of the inverse cocycle $\A^{-1}$ the proof there uses the adjoint cocycle $\A_*$. 

The reason for such a difference is that \cite{park2020quasi} aims to establish a different property called the quasi-multiplicativity. In establishing quasi-multiplicativity of $\A$, it is important to avoid orthogonality between the controlled directions. Hence, it is advantageous to work with the adjoint cocycles $\A_*$ there because $\rho(\mathsf{A}\u,\v) \neq \pi/2$ if and only if $\rho(\u,\mathsf{A}^*\v) \neq \pi/2$. 

On the other hand, for Theorem \ref{thm: 1-transversality}, we had to ensure that the angles between the controlled directions and hyperplanes are bounded away from zero. Hence, the inverse cocycle $\A^{-1}$ was used (instead of the adjoint cocycle $\A_*$) because $\rho(\mathsf{A}\u,\vv) \neq 0$ if and only if $\rho(\u,\mathsf{A}^{-1}\vv)\neq 0$. In fact, the adjoint cocycle $\A_*$ would not have worked for this proof.
\end{rem}

\subsection{The outline and the choice of constants for Theorem \ref{thm: model thm}}
We will now prove Theorem \ref{thm: model thm}. The goal is to construct a periodic point $q \in \Sig$ such that $\mathsf{A}^{n_q}(q)$ maps a cone near $\v_1$ into itself. We will control the $\rho$-norm of $\mathsf{A}^{n_q}(q)$ by using Proposition \ref{prop: AMS}, and apply Proposition \ref{prop: tits} to conclude that $\A^{n_q}(q)$ is proximal. But before starting the proof, let us provide a brief sketch first.

 Let $x\in\Sig$ and $n\in \N$ be given. Since 1-typicality data available to work with are the simple eigendirections of $P:=\A(p)$ associated to the distinguished fixed point $p\in\Sig$, we begin by constructing an orbit segment from a point in $\s^{-n}\Wloc^u(\s^n x)$ to a point in $\Wloc^s(p)$ whose length is uniformly comparable to $n$. Let the corresponding path be denoted by $g\in \glr$. 

Since $x$ and $n$ are arbitrary, it is fair to assume that we do not have any control on the $\rho$-norm of $\gsf$, except for Proposition \ref{prop: AMS} which applies to any matrix in $\glr$; let $\U_{g}$ be the hyperplane corresponding to $g$. 

We now build a path $B:=B_{p,x}$ from $p$ to $x$ while making sure that $v_1$ (the eigenvector of $P$ corresponding to the eigenvalue of the largest modulus) gets mapped uniformly away from $\U_{g}$. This is achieved using the spannability of $\A$, which follows from the transversality established in Theorem \ref{thm: 1-transversality}. We now have a path $gB$ from $p$ to $p$ with some control on its $\rho$-norm. It, however, does not necessarily map a cone $\CC$ around $\v_1$ into itself (i.e., we do not have control over $gBv_1$) yet. So we concatenate $gB$ to a path tracing the holonomy loop $\psi_z$, which will have the effect of applying $P^\ell \psi$ for some $\ell\in \N$ (see the proof of Lemma \ref{lem: 1}), such that the resulting path $\wt{B}$ maps $\CC$ strictly inside itself. This new path $\wt{B}$, although it begins and ends at $p$,  is not of the form $\A^{n_q}(q)$ for some periodic point $q$ yet. So we take the periodic orbit shadowing $\wt{B}$ and show that it has the desired properties listed in Theorem \ref{thm: model thm}. Throughout this entire process, we will have to ensure that the local holonomies that show up (such as $R_{\s^nw}$ from Lemma \ref{lem: connect}) do not destroy the desired properties of the path, and this translates to fine tuning of the parameters such as $\tau$. 

We now introduce relevant constants to be used in the proof.
Let $\ep_1>0$ and $N_1\in\N$ be the constant from Theorem \ref{thm: 1-transversality}, and let $c :=c(\ep_1/2)$ be from Proposition \ref{prop: AMS}. Also, set $$m:=m(\d/3) \text{ and } N_2:=N(\d/3)$$ from Lemma  \ref{lem: rectangle} and  \ref{lem: powers of P} where $\d$ is defined as in \eqref{eq: delta}. Since $\ep_1,N_1$, and $\d$ depend only on the cocycle $\A$, so do $c,m,$ and $N_2$.

Since $\Sig$ is compact, so is the image of $\A$. As the $\rho$-norm can be bounded above using an expression \eqref{eq: rho and norm} involving the operator norm and the conorm,
we can fix $D>0$ serving as a uniform upper bound on the $\rho$-norm of product of matrices of length bounded above by $\max\{N_1,N_2\}$. For instance, we will assume that
\begin{equation}\label{eq: D}
\max\{\|\PP^b\|_\rho, \|\BB\|_\rho,\|\RR\|_\rho\} \leq D
\end{equation}
for any $b \leq \max\{N_1,N_2\}$, and any path $B$ with length bounded  above by $\max\{N_1,N_2\}$, and any local holonomy rectangle $R$ defined as in \eqref{eq: R}. We will also use $D$ as an upper bound for the $\rho$-norm of finite ($\leq 3$ suffices for our purpose) composition of local holonomies. Such a bound will frequently appear in controlling the size of relevant cones via \eqref{eq: rho and cones}.

\subsection{Proof of Theorem \ref{thm: model thm}}\label{subsec: pf thm 3.1}

The beginning of the proof resembles that of Theorem \ref{thm: 1-transversality}. Let $x\in \Sig$ and $n\in \N$ be given. Using the mixing property of $\Sig$, we can find $y \in \s^{-n}\Wloc^u(\s^n x)$ and $n(x) \in \N$ such that $\wt{y}:=\s^{n(x)}y$ belongs to $\Wloc^s(p)$ and that the difference $|n(x)-n|$ is uniformly bounded (by the mixing rate $M \in \N$ of $\Sig$). 
By replacing $\wt{y}$ by $\s^m \wt{y}$ if necessary we may assume that $d(\wt{y},p)\leq 2^{-m}$. Note the difference $|n(x) -n|$ is still uniformly bounded above by some constant $N_0 \in \N$ depending only on the subshift and the cocycle. We set 
\begin{equation}\label{eq: g}
g:=H^s_{\wt{y},p}\A^{n(x)}(y)
\end{equation} and $\U_{g}$ be the hyperplane associated to $g$ from Proposition \ref{prop: AMS}.

The uniform transversality (Theorem \ref{thm: 1-transversality}) of $\A$ applied to $p,y\in \Sig$, $v_1 \in \R^d$, and $\U_{g} \subset \R^d$ gives a path $B:=B_{p,y} =H^s_{\s^{n(y)}b,y} \A^{n(y)}(b)H^u_{p,b}$ with $n(y) \in [0, N_1]$ such that 
$$\rho(\BB\v_1,\UU_{g}) \geq \ep_1.$$
\begin{figure}
\caption{}
\includegraphics[width=16cm]{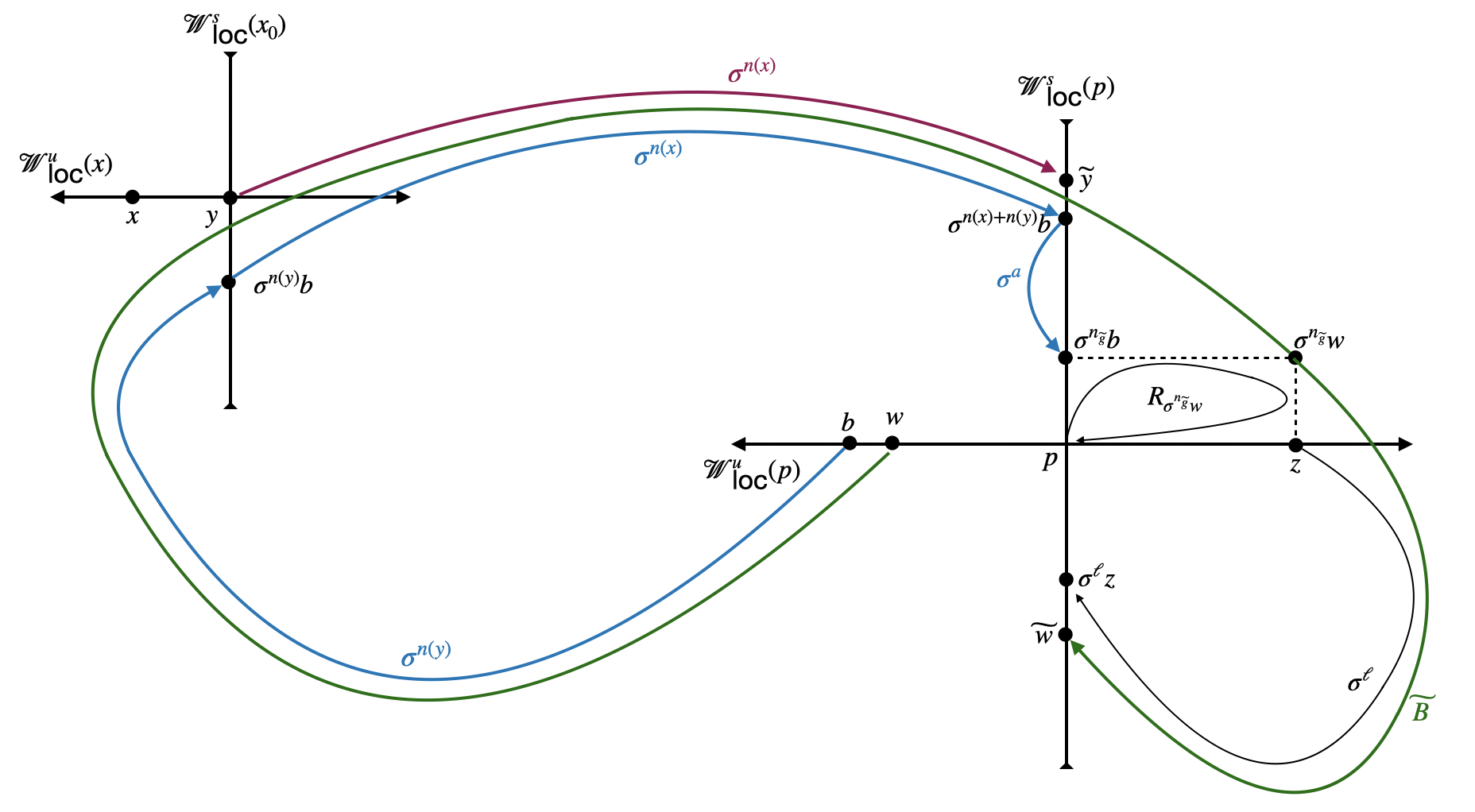}
\label{fig3}
\centering
\end{figure}

Now consider an arbitrary number $\tau>0$.
Assuming that $\tau$ is smaller than $\displaystyle \frac{\ep_1}{2D}$, then $\BB\CC(\v_1,\tau)$ which is a subset of $\CC(\BB\v_1,D\tau)$ belongs to $\CC(\BB\v_1,\ep_1/2)$; note that we have used \eqref{eq: rho and cones} and \eqref{eq: D} here. In particular, we have $\BB\CC(\v_1,\tau) \subseteq \CC(\UU_{g},\ep_1/2)^c$, and hence Proposition \ref{prop: AMS} applies to give
$$\gsf\BB\CC(\v_1,\tau) \subseteq \CC(\gsf\BB\v_1,cD\tau).$$
This gives us a uniform control on the image $\gsf\BB\CC(\v_1,\tau)$ despite the fact that $g$ is constructed with arbitrarily provided data $x$ and $n$. Note also that the concatenation $$gB = H^s_{\s^{n(x)+n(y)}b,p}\A^{n(x)+n(y)}(b)H^u_{p,b}$$ is still a path; see Figure \ref{fig3}.

We now apply Lemma \ref{lem: powers of P} to $\gsf\BB\v_1$ and obtain $a \in [0,N_2]$ such that $\PP^a(\gsf\BB\v_1)$ belongs to $\CC(\v_i,\d/3)$ for some $1\leq i\leq d$. From \eqref{eq: D}, applying $\PP^a$ to the above inclusion gives
$$\PP^a\gsf \BB\CC(\v_1,\tau) \subseteq \CC(\PP^a\gsf \BB\v_1,cD^2\tau).$$

Further assuming that $\tau$ is less than $\displaystyle \frac{\d}{3cD^2}$,  the image of $\CC(\v_1,\tau)$ under 
$$\wt{\gsf}:=\PP^a\gsf \BB=\HH^s_{\s^a \wt{y},p}\AA^{a+n(x)}(y)\BB$$ is contained in $\CC(\v_i,2\d/3)$; that is, $\wt{\gsf}\CC(\v_1,\tau) \subseteq \CC(\v_i,2\d/3).$
Setting 
\begin{equation}\label{eq: tau'}
\tau':=\min\Big\{\frac{\ep_1}{2D},\frac{\d}{3cD^2}\Big\},
\end{equation}
we summarize the construction thus far in the following lemma.
\begin{lem}\label{lem: g_summary}
For any $\tau\in (0,\tau']$, the constructed path $\wt{\gsf}$ maps $\CC(\v_1,\tau)$ inside $\CC(\v_i,2\d/3)$ for some $1\leq i\leq d$. Along its path, $\wt{g}$ shadows the forward orbit of $x$ upto time $n$. Moreover, $\wt{g}$ has length $n _{\wt{g}}:=a+n(x)+n(y)$.
\end{lem}

We now have a path $\wt{g}$ with some control. However, it is not yet enough in that it does not satisfy the proximality condition (first condition of Theorem \ref{thm: model thm}). On the other hand, we know how $\wt{\gsf}$ maps $\CC(\v_1,\tau)$. So we want to further modify the path so that $\CC(\v_1,\tau)$ gets mapped into itself, which then we can apply Proposition \ref{prop: tits} (Tits' criteria) to verify proximality.

Since $a \leq N_2$, $n(y) \leq N_1$, and $|n(x) - n| \leq N_0$, the difference $|n_{\wt{g}}-n|$ can be bounded above by $N_0+N_1+N_2$: 
\begin{equation}\label{eq: M}
|n_{\wt{g}}-n| \leq N_0+N_1+N_2.
\end{equation} 
Note that this upper bound depends only on the base dynamical system $(\Sig,\s)$ and the cocycle $\A$ (i.e., depends on the mixing rate of $\s$ and the constant $\delta$ which depends on the cocycle $\A$) but not on $x$ and $n$.

Now consider $w:=\s^{-n_{\wt{g}}}[\s^{n_{\wt{g}}}b,z] \in \Wloc^u(p)$. For any $\ell\in \N$, we can connect two paths $\wt{g}$ and $P^\ell \psi = H^s_{\s^\ell z}\A^\ell(z) H^u_{p,z}$ using Lemma \ref{lem: connect}:
\begin{align}\label{eq: wtB}
\wt{B}:=H^s_{\wt{w},p}\A^{n_{\wt{g}}+\ell}(w)H^u_{p,w} =P^\ell \psi R_{\s^{n_{\wt{g}}}w} \wt{g}
\end{align}
where $\wt{w}:=\s^{n_{\wt{g}}+\ell}w$. This new path which we call $\wt{B}$ starts and ends at $p$ (via $w$ and $\wt{w}$), and has length $n_{\wt{g}}+\ell$; see Figure \ref{fig3}.
The path $\wt{B}$ has the following property when $\ell$ is large:
\begin{lem}\label{lem: ell}
For any $\tau_1\in (0,\tau']$, $\tau_2>0$, and $\xi>0$, there exists $\ell_1 = \ell_1(\tau_1,\tau_2,\xi) \in \N$ such that for any $\ell \geq \ell_1$, the path $\wt{B}:=H^s_{\wt{w},p}\A^{n_{\wt{g}}+\ell}(w)H^u_{p,w}$ satisfies the following properties:
\begin{enumerate}
\item $\wt{\BB}$ maps $\CC(\v_1,\tau_1)$ into $\CC(\v_1,\tau_2)$, and
\item $\|\wt{\BB} \mid_{\CC(\v_1,\tau_1)}\|_\rho <\xi.$ 
\end{enumerate}
\end{lem}
\begin{proof} 
Consider $\ell(\tau_2)$ defined as in \eqref{eq: 1}. Recalling the expansion $\wt{\BB}=\PP^\ell \Psi \RR_{\s^{n_{\wt{g}}}w}\wt{\gsf}$ from \eqref{eq: wtB}, we claim that $\wt{\BB}\CC(\v_1,\tau_1)\subseteq \CC(\v_1,\tau_2)$ for any $\ell \geq \ell(\tau_2)$.

Indeed, since $\tau_1 $ is less than $\tau'$, Lemma \ref{lem: g_summary} ensures that $\wt{\gsf}$ maps $\CC(\v_1,\tau_1)$ inside $\CC(\v_i,2\d/3)$ for some $i$. Then from Lemma \ref{lem: rectangle}, $\RR_{\s^{n_{\wt{g}}}w}$ does not move any direction off itself more than $\d/3$ in angle because one of its legs connecting $\s^{n_{\wt{g}}}b$ and $p$ has length at most $2^{-m}$. This follows because $d(\s^{n_{\wt{g}}}b,p) \leq d(\s^{n(x)+n(y)}b,p) = d(\wt{y},p) \leq 2^{-m}$. Hence, we have 
$\RR_{\s^{n_{\wt{g}}}w}\wt{\gsf}\CC(\v_1,\tau_1) \subseteq \CC(\v_i,\d)$. Then \eqref{eq: 1} ensures that $\wt{\BB}$ maps $\CC(\v_1,\tau_1)$ into $\CC(\v_1,\tau_2)$.

The second requirement $\|\wt{\BB} \mid_{\CC(\v_1,\tau_1)}\|_\rho <\xi$ can also be met for all $\ell \geq \ell'$ for some $\ell' \in \N$ depending only on $\tau_1$ and $\xi$. Indeed, this is possible because $\v_1$ is the attracting fixed point of $\PP$, and the $\PP^\ell$-term in the expansion \eqref{eq: wtB} will ensure that $\wt{\BB}$ maps $\CC(\v_1,\tau_1)$ as close to $\v_1$ as necessary for sufficiently large $\ell$. 
The required $\ell_1$ in the lemma can then be chosen to be $\max\{\ell(\tau_2),\ell'\}$.
\end{proof}

In view of Proposition \ref{prop: tits} (Tits' criteria) and the above lemma,  for $\ell$ sufficiently large the corresponding path $\wt{B}$ has the desired property listed in Theorem \ref{thm: model thm}. However, what we need is the periodic orbit equipped with such properties. The last step in the proof of Theorem \ref{thm: model thm} is to verify that the periodic orbit $q$ shadowing the path $\wt{B}$ indeed has the desired properties.

Given $\ell \in \N$ sufficiently large, let $q:=q_\ell\in \Sig$ be the periodic point of period $n_q:=n_{\wt{g}}+\ell$ that repeats the first $n_q$ alphabets of $w$. Let 
$$r:=[w,q] \text{ and }\wt{r}:=\s^{n_q}r=[\s^{n_q}w,\s^{n_q}q]=[\wt{w},q].$$ Since $w$ lies on $\Wloc^u(p)$, we can also write $r$ as $[q,p]$; see Figure \ref{fig2}. 
Using these new points, the expression $\A^{n_q}(q)$ can be related to $\wt{B}$ as in the following lemma:

\begin{lem}\label{lem: q-expansion}
Setting $H_1 := H^s_{\wt{r},q}H^u_{\wt{w},\wt{r}}H^s_{p,\wt{w}}$ and $H_2:=H^u_{r,p}H^s_{q,r}$, we have
$$\A^{n_q}(q) = H_1\wt{B}H_2.$$
Moreover, $H_2H_1$ is the rectangle $R_{\wt{r}}$ defined as in \eqref{eq: R}.
\end{lem}
\begin{proof}
From \eqref{eq: same cylinder}, we have
$\A^{n_q}(w) = H^u_{\wt{r}, \wt{w}}H^s_{q, \wt{r}} \A^{n_q}(q) H^s_{r,q}H^u_{w,r}$. Noting that $\wt{B} = H^s_{\wt{w},p}\A^{n_q}(w)H^u_{p,w} $ from \eqref{eq: wtB}, the statement of the lemma follows from the fact that the inverse of $H^{s/u}_{x,y}$ is given by $H^{s/u}_{y,x}$;
see Figure \ref{fig2}.
\end{proof}
\begin{figure}[H]
\caption{}
\includegraphics[width=7.5cm]{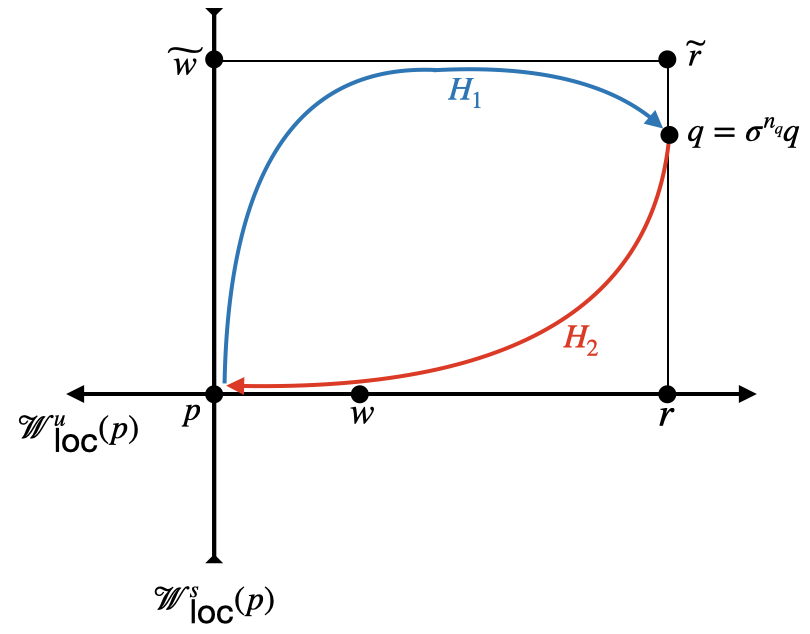}
\label{fig2}
\centering
\end{figure}

Since $H_1$ and $H_2$ are compositions of finite local holonomies, we may assume that $$\max\{\|\HH_1\|_\rho,\|\HH_2\|_\rho\} \leq D$$ where $D$ is the uniform upper bound introduced in \eqref{eq: D}.
Using the construction thus far, we now complete the proof of Theorem \ref{thm: model thm}.
\begin{proof}[Proof of Theorem \ref{thm: model thm}]
For $\tau_0$ appearing in the statement of the theorem, we claim that we can choose $\displaystyle \tau_0:=\frac{\tau'}{6D}$ where $\tau'$ and $D$ are defined in \eqref{eq: tau'} and \eqref{eq: D}. Let $\tau \in (0,\tau_0)$ be an arbitrary number, and $\xi:=\xi(\tau)>0$ be the constant from Proposition \ref{prop: tits}. We will show that there exists  $\ell\in \N$ such that the periodic point $q:=q_\ell$ satisfies (and hence verifies the assumptions in Proposition \ref{prop: tits})
$$\AA^{n_q}(q)\CC(\HH_1\v_1,3\tau)\subseteq \CC(\HH\v_1,\tau) \text{ and }\|\AA^{n_q}(q) \mid_{\CC(\HH_1\v_1,3\tau)}\|_\rho \leq \xi,$$
where $H_1$ is defined in Lemma \ref{lem: q-expansion}.

Let $\ell_1 = \ell_1(\tau', \frac{\tau}{D},\frac{\xi}{D^2})$ be from Lemma \ref{lem: ell}; note that $\ell_1$ depends only on $\tau$. We also let $m_1:=m(\tau'/2)$ be from Lemma \ref{lem: rectangle}. Fix any $\ell \geq \ell_1$ such that $d(\s^\ell z,p) \leq 2^{-m_1}$; such $\ell$ depends only on $\tau$.
Letting $\wt{B}$ be the corresponding path \eqref{eq: wtB} built above for such $\ell$, we first consider the effect of $\HH_2$ on the cone $\CC(\HH_1\v_1,3\tau)$: 
$$\HH_2\CC(\HH_1\v_1,3\tau) \subseteq \CC(\HH_2\HH_1\v_1,3D\tau) \subseteq \CC(\HH_2\HH_1\v_1,\tau'/2) \subseteq \CC(\v_1,\tau').$$ 
The second inclusion used the assumption that $\displaystyle \tau < \tau_0 = \frac{\tau'}{6D}$, and the last inclusion used Lemma \ref{lem: rectangle} and the assumption that $d(\wt{w},p)=d(\s^\ell z,p)\leq 2^{-m_1}$ along with Lemma \ref{lem: q-expansion}.

From the choice of $\ell \geq \ell_1=\ell_1(\tau', \frac{\tau}{D},\frac{\xi}{D^2})$, Lemma \ref{lem: ell} gives
$$\wt{\BB}\HH_2\CC(\HH_1\v_1,3\tau) \subseteq \wt{\BB}\CC(\v_1,\tau') \subseteq \CC(\v_1,\tau/D).$$
Recalling from Lemma \ref{lem: q-expansion} that $\AA^{n_q}(q)=\HH_1 \wt{\BB}\HH_2$, we have
$$\AA^{n_q}(q) \CC(\HH_1\v_1,3\tau) \subseteq \HH_1\CC(\v_1,\tau/D) \subseteq \CC(\HH_1\v_1,\tau).$$
This shows the first requirement of Proposition \ref{prop: tits} that $\AA^{n_q}(q)\CC(\HH_1\v_1,3\tau)\subseteq \CC(\HH\v_1,\tau)$. 
The second requirement on the $\rho$-norm also easily follows. Indeed, Lemma \ref{lem: ell} gives $\|\wt{\BB} \mid_{\CC(\v_1,\tau')}\|_\rho \leq \xi/D^2$, and recall that we have $\max\{\|\HH_1\|_\rho,\|\HH_2\|_\rho\} \leq D$. Putting these together gives the second requirement $\|\AA^{n_q}(q) \mid_{\CC(\HH_1\v_1,3\tau)}\|_\rho \leq \xi$. From Proposition \ref{prop: tits}, this shows that $\A^{n_q}(q)$ is $\tau$-proximal.

The remaining conditions of Theorem \ref{thm: model thm} can easily be checked.
Recalling from \eqref{eq: M} that the difference $|n_{\wt{g}}-n|$ is uniformly bounded above by $N_0+N_1+N_2$ (which depends only on the cocycle) the constant $k(\tau)$ appearing in the statement of the theorem can be set to $k(\tau):=N_0+N_1+N_2+\ell$, which depends only on $\tau$. It is then clear that the period $n_q=n_{\wt{g}}+\ell$ of $q$ belongs in the range $[n,n+k(\tau)]$. Also, since $\wt{B}$ shadows the forward orbit of $x$ upto time $n$ along its path, so does the length $n_q$ periodic orbit of $q$. This completes the proof. 
\end{proof}

We conclude this section by commenting on the uniformity of the constant $k \in \N$ and the constructed periodic point $q \in\Sig$ in Theorem \ref{thm: model thm}.

\begin{rem}\label{rem: uniform constants}
Since Theorem \ref{thm: 1-transversality} is the main ingredient in the proof of Theorem \ref{thm: model thm}, we will comment on its constants $\ep_1,N_1$ and path $B_{x,y}$ first. 
Note that 1-typicality is an open condition because all data attached to $\A$, such as $P$ and $\psi$, vary continuously in $\A$. Hence, the constants $\ep_1,N_1$ can be chosen such that they work uniformly near $\A$. 
Similarly, given $x,y\in \Sig$, a vector $v \in \R^d$ and a hyperplane $V \subset \R^d$, the path $B_{x,y}$ mapping $v$ transverse to $V$ also works uniformly near $\A$. 

Again using the continuity of data attached to $\A$, the constants $\tau_0>0$ and $k(\tau)\in \N$ given $\tau \in (0,\tau_0)$ can be chosen uniformly near $\A$ from Theorem \ref{thm: model thm}. Note, however, that the constructed periodic point $q \in \Sig$ cannot be chosen uniformly near $\A$. 
Indeed, suppose we consider a cocycle $\B$ arbitrarily close to $\A$. Regardless of how close $\B$ is to $\A$, there is no guarantee that $g_\A$ and $g_\B$ defined (for some given $x$ and $n$) as in \eqref{eq: g} would be similar linear maps when the given $n\in \N$ is arbitrarily large. In particular, the hyperplanes $\U_{g_\A}$ and $\U_{g_\B}$ would not necessarily be close, and this will lead to a different choice of the path $B_{p,y}$ for $\A$ and $\B$. This is similar to how quasi-multiplicativity for typical cocycles can be established with constants working uniformly near $\A$, but the connecting word cannot necessarily be chosen uniformly near $\A$; see \cite[Remark 4.20]{park2020quasi}.
\end{rem}

\section{Simultaneous proximality}\label{sec: simult proximality}
The goal of this section is to prove the following generalization of Theorem \ref{thm: model thm} which constructs a periodic point $q \in \Sig$ such that $(\A^{\wedge t})^{n_q}(q)$ is proximal simultaneously for all $1 \leq t \leq d-1$.

\begin{thm}\label{thm: construct proximal general}
Let $\A \colon \Sig \to \glr$ be a typical cocycle. Then there exists $\tau_0 = \tau_0(\A)>0$ such that for any $\tau \in (0,\tau_0)$, there exists $k =k(\tau)\in \N$ with the following property: 
for any $x \in \Sig$ and $n \in \N$ there exists a periodic point $q \in\Sig$ of period $n_q \in [n,n+k]$ such that 
\begin{enumerate}
\item $(\A^{\wedge t})^{n_q}(q)$ is $\tau$-proximal for every $1 \leq t \leq d-1$, and
\item there exists $j\in \N$ such that $\s^j q \in [x]_n$.
\end{enumerate}
\end{thm}

The same argument used in Theorem \ref{thm: model thm} applies here. We will first prove simultaneous transversality for typical cocycles (Theorem \ref{thm: simult transverse}) and then use it together with Proposition \ref{prop: AMS} to control the $\rho$-norm of the path we build just as in Theorem \ref{thm: model thm}.

Before we begin the proof, we set up a few notations and introduce relevant lemmas. 
We will try to maintain the same notation used in Section \ref{sec: proximal construction} as much as possible.
First, we introduce the following general setting that will be assumed throughout the section.\\

\noindent\textbf{General setting:} Let $\k \in \N$.
For each $t \in \{1,\ldots,\kappa\}$, let $\V_t$ be a real vector space of dimension $d_t\in \N$ equipped with an inner product and $\A_t \colon \Sig \to \text{GL}(\V_t)$ a 1-typical cocycle such that there exists a \emph{common} (over all $t$) typical pair $(p,z)$ satisfying the pinching and twisting conditions for each $\A_t$. 
\\

We now make the same simplifications assumed in the proof of Theorem \ref{thm: model thm} that $p\in \Sig$ is a fixed point and that the homoclinic point $z \in \mathcal{H}(p)$ lies on $\Wloc^u(p)$.
As before, we denote the eigenvectors of $P_t:=\A_t(p)$ by $\{v_1^{(t)},\ldots,v_{d_t}^{(t)}\}$, listed in the order of decreasing absolute values for their corresponding eigenvalues. We then define the hyperplanes $$\WW_i^{(t)} :=\{ v_1^{(t)},\ldots, v_{i-1}^{(t)},v_{i+1}^{(t)},\ldots,v_{d_t}^{(t)}\}\subset \V_t .$$
In order to avoid overloading the super/subscripts, throughout the proof, we will suppress the notation ``$(t)$'' whenever the context is clear. For instance, we will often simply write $H^{s/u}$ instead of $H^{s/u,(t)}$ for the canonical holonomies of $\A_t$.  

We also fix a few constants. Let $\d_t>0$ be the corresponding constant for $\Psi^{(t)}$ defined as in \eqref{eq: delta}, and set $\d:=\min\limits_{1 \leq t \leq \k} \d_t$. 
Also for any $\ep>0$, we set $\ell(\ep):=\max\limits_{1\leq t\leq \k} \ell_t(\ep)$ where $\ell_t(\ep)$ is defined as in \eqref{eq: 1} with respect to $\A_t$. Then we have the following property analogous to \eqref{eq: 1}:
\begin{equation}\label{eq: 2}
\PP_t^{\ell(\ep)}\Psi^{(t)} \Big(\bigcup \limits_{i=1}^{d_t}\CC(\v_i^{(t)},\d)\Big) \subseteq  \CC(\v_1^{(t)},\ep) 
\end{equation}
for every $1 \leq t \leq \k$.

The following lemma will play the role of Lemma \ref{lem: powers of P}. It allows us to find a common integer $a\in \N$ such that $\PP_t^a$ simultaneously turns the given direction $\u_t$ close to some $\v_i^{(t)}$ for each $1 \leq t \leq \k$.
\begin{lem}\cite[Lemma 4.22]{park2020quasi}\label{lem: simult turn} Given $\ep>0$, there exists $N = N(\ep)$ such that for any $\u_1 \in \P(\V_1), \ldots,\u_\kappa \in \P(\V_{\kappa})$, there exists $a \in [0,N]$ such that 
$$\PP_t^a \u_t \in \bigcup\limits_{i=1}^{d_t} \CC(\v_i^{(t)},\ep)$$
for every $t \in \{1,\ldots,\kappa\}$.
\end{lem}

\subsection{Simultaneous transversality of typical cocycles}
In what follows, given a path $B$ from $x$ to $y$ we denote by $B_{x,y}^{(t)}$ the cocycle over the path $B$ with respect to $\A_t$.

\begin{thm}[Simultaneous transversality]\label{thm: simult transverse} In the general setting described above, there exist $N_1 \in \N$ and $\ep_1>0$ such that for any $x,y \in \Sig$, any directions $\w_1 \in \P(\V_1),\ldots,\w_\kappa \in \P(\V_\kappa)$, and any hyperplanes $\U_1 \subset \V_1,\ldots, \U_\kappa \subset \V_\kappa$, there exists a path $B_{x,y}$ of length at most $N_1$ such that
$$
\rho(\BB_{x,y}^{(t)}\w_t,\UU_t)>\ep_1
$$
for every $1 \leq t \leq \kappa$. 
\end{thm}

The proof of this theorem is a direct generalization of Theorem \ref{thm: 1-transversality}. Indeed, the only difference is that Lemma \ref{lem: simult turn} replaces the role of Lemma \ref{lem: powers of P}. Hence, we will only briefly sketch the proof below. 

The following lemma which generalizes Lemma \ref{lem: 1} constructs a path from $x$ to $p$ of bounded length which simultaneously turns the given directions $\w_t$ close to the attracting fixed point $\v_1^{(t)}$ of $\PP_t$.

\begin{lem}\label{lem: simult transverse}
For any $\ep>0$ and $k\in \N$, there exists $N \in \N$ such that the following holds: for any $x \in \Sig$ and $\w_1 \in \P(\V_1),\ldots, \w_\kappa \in \P(\V_\kappa)$, there exists a path $x \xrightarrow{H^u} x_0 \xrightarrow{\s^{n(x)}} x_1:=\s^{n(x)}x_0 \xrightarrow{H^s}p$
denoted by $B_{x,p}$ such that its length $n(x)$ is at most $N$, the distance $d(x_1,p)$ is at most $2^{-k}$, and that
$$\rho(\BB_{x,p}^{(t)}\w_t,\v_1^{(t)})\leq \ep$$
for all $1\leq t \leq \k$.
\end{lem}

\begin{proof}[Proof Sketch]
Using $\d>0$ defined as in \eqref{eq: 2}, we set $m:=\max\limits_{1\leq t \leq \k} m_t(\d/2)$ where $m_t(\d/2)$ is defined as in Lemma \ref{lem: rectangle} with respect to $\A_t$.

As in Lemma \ref{lem: 1} (ie, using the mixing rate $M \in \N$ of $\Sig$), we begin by choosing a path $H^s_{w_1,p}\AA^{M}(w_0)H^u_{x,w_0}$ such that $d(w_1,p) \leq 2^{-m}$, and set 
\begin{equation}\label{eq: u_t}
\u_t:=\HH^s_{w_1,p}\AA_t^{M}(x_0)\HH^u_{x,w_0}\w_t.
\end{equation}
Applying Lemma \ref{lem: simult turn} to $\u_1,\ldots,\u_\k$ gives $a \in [0,N']$ where $N':=N(\d/2)$ such that for every $1 \leq t \leq \k$ we have $\PP_t^a \u_t \in \bigcup\limits_{i=1}^{d_t}\CC(\v_i^{(t)},\d/2)$.

We then set $\ell:=\ell(\ep)$ as in \eqref{eq: 2}. By increasing $\ell$ if necessary, we may assume that $d(\s^\ell z,p) \leq 2^{-k}$. Proceeding as in Lemma \ref{lem: 1}, the path $B_{x,p}$ obtained by concatenating $H^s_{\s^a w_1,p}\AA^{M+a}(w_0)H^u_{x,w_0}$ and $H^s_{\s^\ell z,p}\A^\ell(z)H^u_{p,z}$. For such path, we have
\begin{equation}\label{eq: B_x,p 1}
\BB_{x,p}^{(t)}\w_t = \PP_t^{\ell}\Psi^{(t)}\RR_{\wt{x}_0}^{(t)}\PP_t^a \u_t,
\end{equation} which then satisfies $$
\rho(\BB_{x,p}^{(t)}\w_t,\v_1^{(t)}) \leq \ep
$$
for every $1 \leq t\leq \k$. Then setting $N:=M+N'+\ell$ completes the proof.
\end{proof}

Briefly summarizing the construction of Lemma \ref{lem: simult transverse} above, we first built a path from $x$ to $p$ using the mixing property of $(\Sig,\s)$. Under $\PP_t^a$ for some common $a\in \N$, the corresponding directions $\u_t$ were brought close to one of the fixed points $\v_i^{(t)}$ of $\PP_t$. Then by tracing the orbit of $z$, the resulting directions $\PP_t^a\u_t$ were twisted away from $\ww_i^{(t)}$ by the holonomy loop $\Psi^{(t)}$ and mapped close to $\v_1^{(t)}$ by further applying the iterates of $\PP_t$. As a result, the initial directions $\w_t$ were simultaneously turned close to $\v_1^{(t)}$.

The same construction can be applied for hyperplanes.
The following lemma is an extension of Lemma \ref{lem: 2} which constructs a path of uniformly bounded length that simultaneously turns given hyperplanes close to $\ww_1^{(t)}$. As the extension is analogous the above lemma, we omit the proof.

\begin{lem}\label{lem: hyperplane turning}
For any $\ep>0$ and $k \in \N$, there exists $N\in \N$ such that for any $x\in \Sig$ and any hyperplanes $\U_1 \subset \V_1,\ldots,\U_\k \subset \V_\k$, there exists a path $B_{x,p}$ given by $x \xrightarrow{H^u} x_0 \xrightarrow{\s^{n(x)}} x_1 \xrightarrow{H^s}p$ such that its length $n(x)$ is at most $N$, the distance $d(x_1,p)$ is at most $2^{-k}$, and that
$$\rho(\BB_{x,p}^{(t)}\UU_t,\ww_{d_t}^{(t)}) \leq \ep$$
for every $1\leq t\leq \k$.
\end{lem}

We are now ready to begin the proof of Theorem \ref{thm: simult transverse}. Since it resembles the proof of Theorem \ref{thm: 1-transversality} closely, we will again only provide a brief sketch.

\begin{proof}[Proof sketch of Theorem \ref{thm: simult transverse}]
Let $x,y \in \Sig$, directions $\w_1 \in \P(\V_1),\ldots,\w_\kappa \in \P(\V_\kappa)$, and hyperplanes $\U_1\subset \V_1,\ldots, \U_\k \subset \V_\k$ be given.
Let
$$\eta:=\min_{1\leq t \leq \k} \rho(\v_{1}^{(t)},\ww_1^{(t)})>0.$$
Applying Lemma \ref{lem: simult transverse} with $\ep=\eta/4$ and $k=\max\limits_{1 \leq t \leq \k} m_t(\eta/4)$ gives $N \in \N $ and a path $B_1$ from $x$ to $p$ of length $n(x) \leq N$ satisfying $\rho(\BB_{1}^{(t)}\w_t,\v_1^{(t)})\leq \ep$ for all $1\leq t \leq \k$.

Likewise, Lemma \ref{lem: hyperplane turning} applied to the same constants $\ep$ and $k$ with respect to the inverse cocycles $\{\A^{-1}_t\}_{1\leq i \leq k}$ gives a path $B_2^{-}$ from $y$ to $p$ of length $n(y) \leq N$ satisfying $\rho(\BB_2^{(t),-}\UU_t,\ww_1^{(t)})<\ep$ for all $t$. We denote by $B_2$ the path $p \xrightarrow{H^{u}} y_1 \xrightarrow{\s^{n(y)}} y_0 \xrightarrow{H^{s}} y$ obtained by reversing the path $B_2^{-}$.

Setting $r:=\s^{-n(x)}[x_1,y_1]$ and $\wt{r}:=\s^{n(x)+n(y)}r$, Lemma \ref{lem: connect} allows to concatenate $B_1$ and $B_2$ and obtain a new path $B$ from $x$ to $y$. Note that the cocycle $\A_t$ over $B$ is given by 
$$B^{(t)}:=H^s_{\wt{r},y}\A_t^{n(x)+n(y)}(r)H^u_{x,r}=B_2^{(t)} R^{(t)}_{\s^{n(x)}r}B^{(t)}_1.$$
The same argument as in the proof of Theorem \ref{thm: model thm} shows $\rho(\RR^{(t)}_{\s^{n(x)}r}\BB_1^{(t)}\w_t,(\BB_{2}^{(t)})^{-1}\UU_t) \geq \eta/4$. Since the length of $B_2$ is bounded above by $N$, this then implies the existence of $\ep_1>0$ such that $\rho(\BB^{(t)}\w_t,\UU_t)>\ep_1$ for all $1\leq t\leq \k$. 
This completes the proof of Theorem \ref{thm: simult transverse} by setting $N_1 :=2N$.
\end{proof}

\subsection{Proof sketch of Theorem \ref{thm: construct proximal general}}
Instead of proving Theorem \ref{thm: construct proximal general}, we will prove the following theorem which generalizes Theorem \ref{thm: model thm}. By setting $\k := d-1$ and $\A_t := \A^{\wedge t}$, we can easily see that this more general theorem implies Theorem \ref{thm: construct proximal general}.

\begin{thm}\label{thm: proximal construction 2}
In the general setting described above, there exists $\tau_0>0$ 
such that for any $\tau \in (0,\tau_0)$ there exists $k = k(\tau)\in \N$ with the following property: for any $x \in \Sig$ and $n \in \N$ there exists a periodic point $q \in\Sig$ of period $n_q \in [n,n+k]$ such that 
\begin{enumerate}
\item $\A_t^{n_q}(q)$ is $\tau$-proximal for every $1 \leq t \leq \k$, and
\item there exists $j\in \N$ such that $\s^j q \in [x]_n$.
\end{enumerate}
\end{thm}

Recall that the main ingredients of Theorem \ref{thm: model thm} were the uniform transversality (Theorem \ref{thm: 1-transversality}), the control on the $\rho$-norm (Proposition \ref{prop: AMS}), and the control on various directions using Lemma \ref{lem: powers of P}.
As we can imagine, Theorem \ref{thm: simult transverse} and Lemma \ref{lem: simult turn} will respectively replace the role of Theorem \ref{thm: 1-transversality} and Lemma \ref{lem: powers of P}. In fact, these are essentially the only replacements needed in the proof, and hence, we will only provide a brief sketch.

\begin{proof}[Proof sketch of Theorem \ref{thm: proximal construction 2}]
Let $x\in \Sig$ and $n\in \N$ be given. Proceeding as in Subsection \ref{subsec: pf thm 3.1}, we obtain $y \in \s^{-n}\Wloc^u(\s^n x)$ and $n(x) \in \N$ such that  $|n(x) - n|$ is uniformly bounded and that $\wt{y}:=\s^{n(x)}y$ belongs to $\Wloc^s(p)$ and satisfies $d(\wt{y},p)\leq 2^{-m}$ where $m=\max\limits_{1\leq t \leq \k}m_t(\d/3)$. 

We set $g_t:=H^s_{\wt{y},p}\A^{n(x)}_t(y)$ and $\U_{g_t}$ be the corresponding hyperplane in $\V_t$ from Proposition \ref{prop: AMS}. Letting $\ep_1>0$ and $N_1 \in \N$ be the constants from Theorem \ref{thm: simult transverse}, we apply it to $p,y\in \Sig$, $\v_1^{(1)},\ldots,\v_1^{(\k)}$, and $\U_{g_1},\ldots ,\U_{g_\k}$ and obtain a path $B$ from $p$ to $y$ of length $n(y) \in [0, N_1]$ such that 
$$\rho(\BB^{(t)}\v_1^{(t)},\UU_{g_t}) \geq \ep_1$$
for all $1\leq t\leq \k$.

Applying Lemma \ref{lem: simult turn} to $\u_t = \gsf_t\BB^{(t)}\v_1^{(t)}$ gives $a\in [0,N(\d/3)]$ such that
$\PP_t^{a}\u_t \in\bigcup\limits_{i=1}^{d_t} \CC(\v_i^{(t)},\d/3)$  for every $t$. 
Proceeding as in the proof of Theorem \ref{thm: model thm}, there exists $\tau'>0$ such that the path $\wt{B}$ along $p \xrightarrow{H^u} w \xrightarrow{\s^{n_{\wt{g}}+\ell}} \wt{w} \xrightarrow{H^s}p$ constructed as in \eqref{eq: wtB} satisfies the analogous property as Lemma \ref{lem: ell}: for any $\tau_1 \leq \tau'$, $\tau_2>0$, and $\xi>0$, there exists $\ell_1\in \N$ such that for all $\ell \geq \ell_1$,
$$\wt{\BB}^{(t)}\CC(\v_1^{(t)},\tau_1)\subseteq	\CC(\v_1^{(t)},\tau_2) \text{ and }\|\wt{\BB}^{(t)} \mid_{\CC(\v_1^{(t)},\tau_1)}\|_\rho <\xi$$
for all $t$. 

The rest of the proof is essentially identical to the proof of Theorem \ref{thm: model thm}.
Using this property of $\wt{\BB}^{(t)}$, for every $\tau>0$ sufficiently small we can choose $\ell\in \N$ accordingly and construct the periodic point $q := q_\ell\in \Sig$ of period $n_q$ such that $\A_t^{n_q}(q)$ satisfies the listed properties in Theorem \ref{thm: proximal construction 2} for all $t$. 
\end{proof}

\section{Proof of main theorems}\label{sec: proof other thms}

\subsection{Proof of Theorem \ref{thm: simultaneous}} Recalling from \eqref{eq: notation for sing and eigen} that $\chi_1(g)=\log |\text{eig}_1(g)|$ is the logarithm of the spectral radius of $g$, we begin by relating Theorem \ref{thm: model thm} to Theorem \ref{thm: simultaneous} via Proposition \ref{prop: property of proximal map}.

\begin{prop}\label{prop: norm and eigen} Suppose $\A \colon \Sig \to \glr$ is a fiber-bunched cocycle with the following property: there exist $\ep>0$ and $k\in \N$ such that for any $x\in \Sig$ and $n\in \N$ there exists a periodic point $q \in\Sig$ of period $n_q \in [n,n+k]$ such that 
\begin{enumerate}
\item $\A^{n_q}(q)$ is $\ep$-proximal, and
\item there exists $j\in \N$ such that $\s^jq \in [x]_n$.
\end{enumerate}
Then there exists a constant $C>0$ depending only on $\A$, $\ep$, and $k$ such that 
$$\Big|\log\|\A^n(x)\| - \chi_1(\A^{n_q}(q)) \Big| \leq C.$$
\end{prop}

\begin{proof}
The triangle inequality gives
$$\Big|\log\|\A^n(x)\| - \chi_1(\A^{n_q}(q)) \Big| \leq \Big|\log\|\A^n(x)\| - \log\|\A^{n_q}(q)\| \Big|+ \Big|\log\|\A^{n_q}(q)\| - \chi_1(\A^{n_q}(q)) \Big|.$$
Due to (2) and the fact that $|n_q - n| \leq k$, the first term admits a following upper bound:
$$ \Big|\log\|\A^n(x)\|-\log\|\A^{n_q}(q)\|\Big| \leq k\cdot \max\big\{\log M,-\log m\big\}+\log C',$$
where $M:=\max\limits_{x \in \Sig}\|\A(x)\|$, $m = \min\limits_{x \in \Sig}\|\A(x)^{-1}\|^{-1}$, and $C'$ is the constant from the bounded distortion of $\A$. 

Since $\A^{n_q}(q)$ is $\ep$-proximal, it follows from Proposition \ref{prop: property of proximal map} that the second term is also bounded above by a uniform constant, completing the proof. 
\end{proof}

\begin{proof}[Proof of Theorem \ref{thm: simultaneous}]
The proof amounts to combining Theorem \ref{thm: construct proximal general} and Proposition \ref{prop: norm and eigen}.
Let $x \in \Sig$ and $n\in \N$ be given. Theorem \ref{thm: construct proximal general} gives $\ep>0$ (i.e., we fix some $\ep \in (0,\tau_0)$), $k = k(\ep)\in \N$, and a common periodic point $q\in \Sig$ of period
$n_q \in [n,n+k]$ such that $\s^jq \in [x]_n$ for some $j\in \N$ and that $$(\A^{\wedge t})^{n_q}(q) \text{ is }\ep\text{-proximal for all }t\in \{1,\ldots,d-1\}.$$

Applying the arguments of Proposition \ref{prop: norm and eigen} to each $\A^{\wedge t}$ gives $C_t>0$ such that 
\begin{equation}\label{eq: norm and eigen exterior}
\Big|\log\|(\A^{\wedge t})^n(x)\| - \chi_1((\A^{\wedge t})^{n_q}(q)) \Big| \leq C_t.
\end{equation}
Note also that this inequality is true when $t = d$ for some $C_d>0$ since $\A^{\wedge d}$ acts by scalar multiplication on a 1-dimensional vector space $(\R^d)^{\wedge d}$ and $|n_q-n|$ is bounded above by $k$. Also, we set $C_0 = 0$.

Note $\log\|(\A^{\wedge t})^n(x)\| = \mu_1((\A^{\wedge t})^n(x)) = \sum\limits_{i=1}^t \mu_i(\A^n(x))$, and likewise for $\chi_1$.
For every $1 \leq i \leq d$, the triangle inequality applied to \eqref{eq: norm and eigen exterior} for $t= i$ and $t=i-1$ gives 
$$\Big|\mu_i(\A^n(x)) - \chi_i(\A^{n_q}(q))\Big| \leq C_i+C_{i-1}.$$
This completes the proof by taking $C:= \sqrt{d} \cdot \max\limits_{1\leq i\leq d}( C_i+C_{i-1})$.
\end{proof}
\begin{rem}
We note that Remark \ref{rem: uniform constants} also applies to Theorem \ref{thm: simultaneous}; that is, the constants $C>0$ and $k\in \N$ may be chosen to work uniformly near $\A$, but the same is not true for the constructed periodic point $q \in \Sig$.
\end{rem}

\subsection{Dominated splitting and the proof of Theorem \ref{thm: main}}
Recalling that $m(g)=\|g^{-1}\|^{-1} = \a_d(g)$ is the conorm of $g \in \glr$, we begin by formally defining dominated cocycles.

\begin{defn}\label{defn: domination}
A continuous cocycle $\A \colon \Sig \to \glr$ is \emph{dominated} if there exist two $\A$-invariant continuous bundles $E$ and $F$ over $\Sig$ such that $E_x\oplus F_x = \R^d$ for every $x\in \Sig$ and there exist constants $C>0$ and $\tau \in (0,1)$ such that 
$$\frac{\|\A^n(x)\mid_{F_x}\|}{m(\A^n(x)\mid_{E_x})} \leq C\tau^n$$
for every $x\in \Sig$ and $n\in \N$. We say that $E$ \emph{dominates} $F$ and that $\A$ has the \emph{$i$-th dominated splitting} where $i$ is the dimension of the bundle $E$.
\end{defn}

Recalling from \eqref{eq: notation for sing and eigen} that $\mu_1(g) \geq \ldots \geq \mu_d(g)$ are the logarithm of the singular values of $g\in \glr$, the following result of Bochi and Gourmelon \cite{bochi2009some} provides a characterization for the $i$-th domination.
\begin{prop}\cite[Theorem A]{bochi2009some}\label{prop: dom split} Given a compact metric space $X$ and a linear cocycle $\A \colon X \to \glr$, $\A$ has $i$-th dominated splitting if and only if there exist $C_1,C_2>0$ such that 
$$(\mu_i-\mu_{i+1})(\A^{n}(x)) \geq C_1n-C_2$$
for all $x\in \Sig$ and $n\in \N$.
\end{prop}

We are now ready to prove Theorem \ref{thm: main}. In order to exploit the assumption \eqref{eq: assumption main} of Theorem \ref{thm: main}, we will use the fact \eqref{eq: Lyap periodic point} that for any periodic point $q \in \Sig$ of period $n_q\in \N$, its Lyapunov exponents are given by
$$\vec{\lambda}(q)=\frac{1}{n_q}  \vec{\chi}(\A^{n_q}(q)) = \frac{1}{n_q}  \big(\log|\eig_1(\A^{n_q}(q))|,\ldots,\log|\eig_d(\A^{n_q}(q))|\big).$$

\begin{proof}[Proof of Theorem \ref{thm: main}]
Let $c>0$ be the constant appearing in the assumption \eqref{eq: assumption main} from the statement of Theorem \ref{thm: main}.
 It suffices to verify Bochi-Gourmelon's criterion.
 
From Theorem \ref{thm: simultaneous}, there exists $C>0$ such that for any $x\in \Sig$ and $n\in \N$, there exists a periodic point $q \in \Sig$ of period $n_q \geq n$ such that
$$\Big\|\vec{\mu}(\A^n(x)) - \vec{\chi}(\A^{n_q}(q))\Big\| \leq C.$$
It then follows from \eqref{eq: Lyap periodic point} that
\begin{align*}
(\mu_i-\mu_{i+1})(\A^{n}(x)) & \geq n_q\cdot (\lambda_i-\lambda_{i+1})(\A^{n_q}(q))-2C \\
&\geq cn-2C
\end{align*}
where the second inequality is due to the assumption \eqref{eq: assumption main} and the fact that $n_q\geq n$. This verifies the assumptions in Proposition \ref{prop: dom split} and establishes the $i$-th dominated splitting.
\end{proof}

\begin{rem}
The conclusion of Theorem \ref{thm: main} is likely to be true with assumptions weaker than the typicality assumption. Our method, however, relies heavily on the typicality assumption, and a new method will have to be deployed in order to address more general class of fiber-bunched cocycles than typical cocycles.
\end{rem}

\subsection{Subadditive thermodynamic formalism and the proof of Theorem \ref{thm: subadditive}}\label{subsec: subadditive thermo}
For any cocycle $\A \colon \Sig \to \glr$ and $s \geq 0$, the \emph{s-singular value potential} $\Phi_\A^s := \{\log \vp_n^s(\cdot)\}_{n\in \N}$ is a sequence of continuous functions on $\Sig$ defined by 
$$\vp^s_n(x):=\vp^s(\A^n(x))$$
where $\vp^s$ is the \textit{singular value function} defined as
$$\vp^s(g) := \a_1(g) \ldots \a_{\lfloor s \rfloor}(g) \a_{\lfloor s \rfloor+1}(g)^{s - \lfloor s \rfloor}$$
for $s \in [0,d]$ and $\vp^s(g) := |\det(g)|^{s/d}$ for $s>d$.

Such potentials are \emph{subadditive}, and the theory of subadditive thermodynamic formalism applies. For instance, the \emph{subadditive variational principle} (see \cite{cao2008thermodynamic}) states that 
$$P(\Phi_\A^s) = \sup\limits_{\mu\in \M(\s)} \Big\{h_\mu(\s)+\lim\limits_{n\to \infty}\frac{1}{n}\int \log \vp_n^s(x)\,d\mu\Big\}.$$
Any invariant measures achieving the supremum are called the \emph{equilibrium states} of $\Phi_\A^s$. We note that the subadditive variational principle remains valid if the supremum is taken over all ergodic measures $\E(\s)$ instead of all invariant measures $\M(\s)$.

The author showed in \cite{park2020quasi} that for typical cocycles $\A$, the singular value potential $\Phi_\A^s$ has a unique equilibrium state $\mu_{\A,s}$ for all $s \geq 0$. Moreover, $\mu_{\A,s}$ has the \emph{subadditive Gibbs property}: there exists $C\geq 1$ such that for any $x\in \Sig$ and $n\in\N$, we have
$$C^{-1}\leq \frac{\mu_{\A,s}([x]_n)}{e^{-nP(\Phi_\A^s)}\vp_n^s(x)} \leq C.$$
Such results generalize analogous previous results for additive potentials \cite{bowen1974some} and irreducible locally constant cocycles \cite{feng2010equilibrium}.

Let $s = 1$ for now, and consider two typical cocycles $\A,\B$ and their unique equilibrium states $\mu_\A:=\mu_{\A,1}$ and $\mu_\B:=\mu_{\B,1}$. From the Gibbs property, for any $x\in \Sig$ and $n\in \N$ we have
$$\frac{\mu_{\A}([x]_{n})}{e^{-nP(\Phi_\A)}\|\A^{n}(x)\|} \asymp 1 \asymp \frac{\mu_{\B}([x]_{n})}{e^{-nP(\Phi_\B)}\|\B^{n}(x)\|}.$$
If we further suppose that $\mu_\A$ and $\mu_\B$ are the same, then for any $x\in \Sig$ whose top Lyapunov exponents $\lambda_1(\A,x)$ and $\lambda_1(\B,x)$ both exist, we have
$\lambda_1(\A,x) - \lambda_1(\B,x) = P(\Phi_B) - P(\Phi_\A).$
Since periodic points are Lyapunov regular, we have
$$\lambda_1(\A,p) - \lambda_1(\B,p) = P(\Phi_B) - P(\Phi_\A)$$
for every periodic point $p \in \Sig$. In the following proof of Theorem \ref{thm: subadditive}, we prove the converse of this statement when $B$ is a small perturbation of $\A$.

\begin{proof}[Proof of Theorem \ref{thm: subadditive}]
Let $\A \colon \Sig \to \glr$ be a 1-typical cocycle and $\B \colon \Sig \to \glr$ a small perturbation of $\A$ satisfying the assumption \eqref{eq: A,B}; that is, 
top Lyapunov exponents at every periodic point $p\in \Sig$ satisfy
$$\lambda_1(\A,p) - \lambda_1(\B,p) = c$$
for some $c\in \R$.

We will first show that this implies that for any ergodic measure $\mu\in \E(\s)$, the difference between its top Lyapunov exponents with respect to $\A$ and $\B$ is also equal to $c$.
This is because for any $\d>0$, we can choose $x \in \Sig$ (from a full $\mu$-measure set) and $n\in \N$ such that $$\max\Big\{\Big|\frac{1}{n}\log\|\A^n(x)\| - \lambda(\A,\mu)\Big|,\Big|\frac{1}{n}\log\|\B^n(x)\| - \lambda(\B,\mu)\Big|\Big\} \leq \d.$$
Since $\B$ is a small perturbation of $\A$, they both have a common typical pair. In particular, applying Theorem \ref{thm: proximal construction 2} to $\A_1 = \A$ and $\A_2 = \B$ gives $\ep>0$ (by fixing some $\ep \in (0,\tau_0)$), $k = k(\ep) \in \N$, and a periodic point $q \in \Sig$ such that both $\A^{n_q}(q)$ and $\B^{n_q}(q)$ are $\ep$-proximal and that the difference $|n_q-n|$ is bounded above by $k$. Following along the proof of Theorem \ref{thm: simultaneous} (i.e., applying Proposition \ref{prop: norm and eigen}) gives a uniform constant $C>0$ such that
$$\max\Big\{\Big|\log \|\A^n(x)\| - \log|\eig_1(\A^{n_q}(q))|\Big| ,\Big|\log \|\B^n(x)\| - \log|\eig_1(\B^{n_q}(q))|\Big| \Big\}\leq C.$$
Recalling that $\displaystyle\lambda_1(\A,q) = \frac{1}{n_q} \log|\eig_1(\A^{n_q}(q))|$ and likewise for $\lambda_1(\B,q)$, it follows that
$$\max\Big\{\Big|\lambda_1(\A,\mu)-\frac{n_q}{n}  \lambda_1(\A,q)\Big|,\Big|\lambda_1(\B,\mu)-\frac{n_q}{n}  \lambda_1(\B,q) \Big|\Big\}\leq \frac{C}{n}+\d.$$
By choosing $\d$ arbitrarily close to $0$ and choosing $x\in \Sig$ and $n\in \N$ accordingly so that $n \to \infty$, top Lyapunov exponents $\lambda_1(\A,q)$ and $\lambda_1(\B,q)$ of the common periodic point $q$ simultaneously approximate the Lyapunov exponents $\lambda_1(\A,\mu)$ and $\lambda_1(\B,\mu)$ of $\mu$.
Here, we have used the fact that $n_q\in [n,n+k]$ which implies that $n_q/n \to 1$ as $n \to \infty$. 
Then the assumption \eqref{eq: A,B} implies that
$$\lambda_1(\A,\mu)- \lambda_1(\B,\mu) =c$$ for all ergodic measures $\mu \in \E(\s)$.
The subadditive variational principle then implies that 
$$
P(\Phi_\A) = \sup\limits_{\mu \in \E(\s)} h_{\mu}(\s)+\lambda_1(\A,\mu)
=\sup\limits_{\mu \in \E(\s)} h_{\mu}(\s)+\lambda_1(\B,\mu)+c=P(\Phi_\B)+c$$
and that the unique equilibrium state $\mu_\A$ for $\Phi_\A$ is also an equilibrium state for $\Phi_\B$. From their uniqueness, $\mu_\A$ must coincide with $\mu_\B$.
\end{proof}

We end this subsection with a few comments on Theorem \ref{thm: subadditive} and its proof above. First, while Theorem \ref{thm: subadditive} is stated for norm potentials $\Phi_\A$ and the corresponding equilibrium states $\mu_\A$, analogous proof applies to singular value pontentials $\Phi_\A^s$ for any $s\geq 0$. In fact, if $\A,\B$ are sufficiently close typical cocycles and there exists $c\in \R$ such that $$\sum\limits_{i=1}^k \big(\lambda_i(\A,p) -\lambda_i(\B,p)\big) = c$$ for all periodic points $p\in \Sig$, then the unique equilibrium state $\mu_{\A,k}$ of $\Phi_\A^k$ coincides with $\mu_{\B,k}$. In particular, if there exists $c_i \in \R$ for each $1 \leq i \leq d$ such that
$$\lambda_i(\A,p) -\lambda_i(\B,p)=c_i$$
for all periodic points $p \in \Sig$ and $1 \leq i \leq d$, then $\mu_{\A,s}$ coincides with $\mu_{\B,s}$ for all $s\geq 0$.

Another point to note from the proof of Theorem \ref{thm: subadditive} is the simultaneous approximation of the Lyapunov exponents $\vec{\lambda}(\A,\mu)$ and $\vec{\lambda}(\B,\mu)$ by common periodic points. Kalinin \cite{kalinin2011livvsic} showed that for any \hol continuous cocycles $\A \colon \Sig \to \glr$ and any ergodic measure $\mu\in \E(\s)$, there exists a sequence of periodic points whose Lyapunov exponents approach the Lyapunov exponent $\vec{\lambda}(\A,\mu)$. Above proof of Theorem \ref{thm: subadditive} shows that for any two sufficiently close typical cocycles $\A$ and $\B$, we can choose a sequence of periodic points whose Lyapunov exponents with respect to both $\A$ and $\B$ simultaneously approach $\vec{\lambda}(\A,\mu)$ and $\vec{\lambda}(\B,\mu)$, respectively.

\subsection{Lyapunov spectrums and the proof of Theorem \ref{thm: spectrum}}\label{subsec: spectrum}
For any cocycle $\A \colon \Sig \to \glr$, there are various Lyapunov spectrums associated to it. The pointwise Lyapunov spectrum $\Omega_\A$ and the Lyapunov exponents over periodic points $\Omega_p$ defined in the introduction are examples of such spectrums. Moreover, we denote the set of Lyapunov exponents $\vec{\lambda}(\mu)$ over all ergodic measures $\mu\in \E(\s)$ by $\Omega_e$. Without any further assumptions on the cocycle, the only relations among them are
$$\Omega_p \subseteq \Omega_e \subseteq \Omega_\A.$$

Kalinin \cite{kalinin2011livvsic} showed that when $\A$ is \hol continuous, then $\Omega_e \subseteq \ol{\Omega}_p$. When $\A$ is fiber-bunched and typical, then the author \cite{park2020quasi} showed that $\Omega_\A$ is closed and convex (see also earlier result of Feng \cite{feng2009lyapunov} concerning top Lyapunov exponents of irreducible locally constant cocycles). In particular, this implies that $\ol{\Omega}_p$ is a subset of $\Omega_\A$ for typical cocycles. Theorem \ref{thm: spectrum} shows that $\ol{\Omega}_p$ in fact coincides with $\Omega_\A$ when $\A$ is typical.

\begin{proof}[Proof of Theorem \ref{thm: spectrum}]
The proof resembles that of Theorem \ref{thm: subadditive}.
Let $x\in \Sig$ be an arbitrary Lyapunov regular point. Recalling the notations from \eqref{eq: notation for sing and eigen}, for any $\d>0$ we can choose $n\in \N$ such that 
$$\Big\| \frac{1}{n}\vec{\mu}(\A^n(x))- \vec{\lambda}(x)\Big\| \leq \d.$$
Theorem \ref{thm: simultaneous} then gives constants $C>0$ and $k\in \N$, and a periodic point $q\in \Sig$ of period $n_q \in [n,n+k]$ such that
$$\Big\|\vec{\mu}(\A^n(x)) - \vec{\chi}(\A^{n_q}(q))\Big\| \leq C.
$$
It then follows that 
$$\Big\|\vec{\lambda}(x) - \frac{n_q}{n} \vec{\lambda}(q)\Big\| \leq \frac{C}{n}+\d.
$$

By choosing $\d$ arbitrarily close to $0$ and choosing $n\in \N$ accordingly so that $n \to \infty$, the Lyapunov exponent $\vec{\lambda}(q)$ of the periodic point $q \in \Sig$ limits to $\vec{\lambda}(x) \in \Omega_\A$.
\end{proof}

\bibliographystyle{amsalpha}
\bibliography{proximal}

\end{document}